\documentclass[12pt]{amsart}

\usepackage{amssymb}
\usepackage{amsmath}
\usepackage{amsfonts}
\usepackage[all]{xy} 

\newcommand{\disk}{\ensuremath{\mathbb{D}} } 

\newcommand{\sphere}{\overline{\mathbb{C}}}

\theoremstyle{plain}
        \newtheorem{theorem}{Theorem}[section]
        \newtheorem{lemma}[theorem]{Lemma}
        \newtheorem{proposition}[theorem]{Proposition}
        \newtheorem{corollary}[theorem]{Corollary}

\theoremstyle{definition}
        \newtheorem{definition}[theorem]{Definition}

\theoremstyle{remark}
    \newtheorem{remark}[theorem]{Remark}

\numberwithin{equation}{section} 
\numberwithin{figure}{section} 

\usepackage{fullpage}

\author{Eric Schippers}
\title{Conformal invariants associated with quadratic differentials}
\begin{document}
\begin{abstract}  In \cite{Nehari_some_inequalities} Z. Nehari
 developed a general technique for obtaining inequalities for conformal
 maps and domain functions  from
 contour integrals and the Dirichlet principle.   Given a
 harmonic function with singularity on a domain $R$, it
 associates a monotonic functional of
 subdomains $D \subseteq R$.
 In the case that $R$ is conformally equivalent to a disk, we extend Nehari's method
 by associating a functional to any quadratic differential on $R$ with specified
 singularities.  Nehari's method corresponds to the special case that the quadratic
 differential is of the form $(\partial q)^2$ for a singular harmonic function $q$
 on $R$.  Besides being more general, our formulation
 is conformally invariant, and has a particularly elegant equality statement.
 As an application we give a one-parameter family of monotonic, conformally
 invariant functionals which correspond to growth theorems
 for bounded univalent functions.  These generalize and interpolate the Pick growth theorems,
 which appear in a conformally invariant form equivalent to a two-point distortion theorem
 of W. Ma and D. Minda.
\end{abstract}

\maketitle
\begin{section}{Introduction}
\begin{subsection}{Statement of results}
 Let $D_1$ be conformally equivalent to the unit disk $\disk = \{ z\,:|z|<1 \}$
 and $D_2 \subseteq D_1$ be a sufficiently regular simply connected domain (we will
 make this precise below).   Let
 $Q(z)dz^2$ be a quadratic differential such that the boundary of $D_1$ is
 network of trajectories of $Q(z)dz^2$.  Assume that $Q(z)dz^2$ has finitely
 many poles in $D_1$.
 In this paper, we define a functional $m(D_1,D_2,Q(z)dz^2)$ with the following
 properties:
 \begin{enumerate}
  \item The functional $m(D_1,D_2,Q(z)dz^2)$ is conformally invariant (Theorem \ref{th:conformal_invariance}).
  \item The functional is monotonic, in the sense that
  \[  D_3 \subseteq D_2 \subseteq D_1 \Rightarrow
   m(D_1,D_3,Q(z)dz^2) \leq m(D_1,D_2,Q(z)dz^2)  \] (Corollary \ref{co:monotonicity}).
  \item The functional is bounded by zero: $m(D_1,D_2,Q(z)dz^2) \leq 0$ and equality
  holds for $D_2=D_1$ (Theorem \ref{th:positivity}).  Furthermore, equality holds if and only if $D_1 \backslash D_2$
  consists of trajectories of the quadratic differential $Q(z)dz^2$ (Theorem \ref{th:equality}).
 \end{enumerate}
 This can be used to give an infinite series of families of inequalities
 for bounded univalent functions, one for each quadratic differential.
 Furthermore, these inequalities are given in terms
 of monotonic coefficient functionals.

 The invariants are constructed using a generalization of a technique of Nehari.  In this technique, the
 functionals were generated not by a quadratic differential but rather by
 a harmonic function with singularities.  Our approach results in several
 improvements which we list here, to be explained in the next section.
 \begin{enumerate}
  \item The role of quadratic differentials is made clear (they generate
  the functional, rather than appear only in a necessary condition for
  extremality).
  \item The invariants are manifestly conformally invariant (Theorem \ref{th:conformal_invariance}).
  \item The set of functionals is significantly larger (one for each quadratic differential).
  \item We explicitly identify the remainder term (that is, the difference
  between the values of the functional on nested domains) itself as a
  conformal invariant associated to a quadratic differential (Theorem \ref{th:monotonicity_with_remainder}).
 \end{enumerate}

 The reader will likely observe that the results of this paper have straightforward
 generalizations to
 hyperbolic Riemann surfaces.  In fact, the conformal invariants
 can be shown to be a special case of a modular invariant
 on Teichm\"uller space, as will be shown in a future publication.  Here we have restricted to two simply connected domains in
 order to make
 a clear and comprehensive connection with bounded univalent functions.  We will treat
  generalizations in future publications.

  For every quadratic differential admissible for the disk, the general theorem produces a distinct,
  sharp estimate on a monotonic functional on  the set of bounded univalent functions.  Furthermore this
  functional is automatically conformally invariant.  Not unexpectedly, the derivation of an expression for the
  functional in terms of a conformal map requires a lengthy computation for each choice of
  quadratic differential.  We will give a
  collection of examples with K. Mather \cite{MatherSchippers}.  To illustrate
  the method here, we give a new one parameter
  family of growth theorems, generalizing
  the Pick growth theorem for bounded univalent function.  These inequalities in fact interpolate the
  upper and lower bounds by allowing the zero of the quadratic differential to vary over the boundary.
  It also leads naturally to a monotonic form of a two-point distortion theorem of Ma and Minda by
  placing the single and double pole in arbitrary location.  The upper and lower bound are obtained
  by choosing the zero at one of two possible endpoints of the hyperbolic geodesic through the poles.
\end{subsection}
\begin{subsection}{Literature and context}
 In \cite{Nehari_some_inequalities} Nehari systematically associated positive monotonic
 quantities to harmonic functions with singularities, as a method of obtaining inequalities
 in function theory.  This can be thought
 of as a variant on the method of contour integration, with unifying emphasis on the Dirichlet principle.
 One advantage of the Nehari's method
 is that it produces inequalities for higher-order derivatives of mapping functions or
 domain functions (such as Green's function or Neumann's function) fairly naturally,
 by choosing the order of the singularity.  The second author used this in
 \cite{Schippers_conf_inv_analyse} \cite{Schippers_conf_inv} in order to obtain
 estimates on higher-order conformal invariants.

 However, Nehari's method skips every other order of differentiation.  Furthermore
 it is not manifestly conformally invariant; the conformal invariance must be
 imposed in ad hoc ways \cite{Schippers_conf_inv_analyse}.
 This paper remedies both problems, by identifying
 the role of quadratic differentials in Nehari's method,
 while at the same time extending the method.  Nehari's method corresponds
 to the special case that the quadratic differential is a perfect square
 of the form $(\partial q)^2$ for a singular harmonic function $q$.  As mentioned
 above, we restrict
 to the simply-connected case in this paper.

 Conformal invariants are associated with several methods for producing inequalities
 for mapping functions or domain functions.
 Three examples are capacitance (e.g. A. Baernstein and A. Solynin \cite{Baernstein_Solynin},
 P. Duren and J. Pfaltzgraff \cite{DurenPfaltzgraff}; or the monograph of
 V. Dubinin \cite{Dubinin_book}), extremal length (see the monographs of
 J. Jenkins \cite{Jenkins_book}, G. Kuz'mina \cite{Kuzmina_mono} and A. Vasil'ev \cite{Vasilev_monograph})
 or Dirichlet energy (see Nehari \cite{Nehari_some_inequalities} or Dubinin \cite{Dubinin_book}),
 Of course, given the connections between these conformal invariants \cite{Ahlfors_conf_inv}, there is
 not always a clear
 boundary between these methods.

 Given the Ahlfors-Beurling theorem connecting Dirichlet energy and extremal length
 \cite[Theorem 4-5]{Ahlfors_conf_inv}, one might expect that the conformal
 invariance of Nehari's functionals is automatic.  This is not the case,
 since the
 harmonic function has a singularity and the theorem does not apply.
 Quadratic differentials of order two are associated with reduced modules; however,
 as far as the authors are aware there are no examples of conformally invariant
 reduced modules of higher order.  In the literature one often finds the restriction
 to quadratic differentials with poles of order two or lower, even when
 conformal invariance is not explicitly demanded.  In \cite{MatherSchippers}
 we will give explicit computations of the conformal invariants given here
 for various orders of the poles.
 Reduced modules of higher order
 are implicit
 in the so-called general coefficient theorem and the length-area method in general (see e.g.
 Jenkins \cite{Jenkins_certain_coefficients},  \cite{Jenkins_book} and
 Schmidt \cite{Schmidt}); however they
 are not manifestly conformally invariant.
\end{subsection}
\end{section}
\begin{section}{Quadratic differentials and simple domains}
\begin{subsection}{Quadratic differentials: definition and terminology}
 In this section we review some basic facts about quadratic differentials.
 In order to
 properly formulate conformal invariance and admissibility of quadratic differentials,
 it is necessary to use the formalism of the double and
 boundary of bordered surfaces.  This is also required since we define the invariants
 by lifting to branched double covers of the disk, which cannot in general be regarded as
 subsets of $\mathbb{C}$.

 Every Riemann surface in this paper $R$ will
 be a bordered Riemann surface in the sense of Ahlfors and Sario \cite{Ahlfors_Sario_book}.
 For our purposes we may state this as follows.   $R$ has a double $S$; for each point $p \in \partial R \subset S$ there is an
 open set $U$ of $S$ and local biholomophic parameter $\phi: U \rightarrow \mathbb{C}$ such
 that $\phi(U)$ is
 an open subset of $\mathbb{C}$ such that (1) $\phi(U \cap \partial R)$ is an open interval
 of $\mathbb{R}$,  (2) if $q \in U$ then the conjugate point $q^*$ is in $U$ and
 (3) $\phi(q^*) = \overline{\phi(q^*)}$.  Thus $\phi$ is obtained from $\left. \phi \right|_{U \cap R}$
 by Schwarz reflection.  We call such a chart a boundary coordinate and the image points $z \in \phi(U)$
 (as a function of points on the surface $R$) as a boundary parameter.  For an arbitrary
 chart $\phi$ on $R$ we refer to images points as local parameters.  Observe that the
 boundary of $R$ is an analytic curve in the double.

 A quadratic differential on a bordered Riemann surface $R$ is a meromorphic $2$-differential on $R$;
 locally it can be written $Q(z)dz^2$ for a local parameter $z$ \cite{Lehtobook}.

  \begin{definition}
  Let $\alpha$ be a quadratic differential on a bordered Riemann surface $R$ and
  let $\Gamma:(a,b) \rightarrow R$ be a smooth curve.   We say that
  $\Gamma$ is a trajectory if given any point $p$ in the image of $\Gamma$
  and local parameter $z$, the local representations $z =\gamma(t)$ of $\Gamma$
  and $Q(z)dz^2$ of $\alpha$ satisfy
  \begin{equation} \label{eq:trajectory_condition}
    Q(\gamma(t)) \cdot {\gamma}'(t)^2 <0.
  \end{equation}
  If $\Gamma:[a,b] \rightarrow \mathbb{C}$ is a continuous
 curve whose restriction to $(a,b)$ is a trajectory, then we will also refer to $\Gamma$ as a
 trajectory.  Similarly for half-open intervals.

 If $\Gamma:(a,b) \rightarrow \partial R$ is a smooth curve, then we say that $\Gamma$ is a
 trajectory if $\alpha$ extends complex analytically to an open neighbourhood of the image of $\Gamma$
 in the double of $R$ and $\Gamma$ is a trajectory in the sense above on the double.
 Similarly for closed or half-open intervals.
 \end{definition}
 It is easily checked
 that this definition is independent of the choice of local parameter.
 We will also refer to the image of $\Gamma$ as a trajectory.

 We say that a quadratic
 differential is admissible for $R$ if, for all but finitely many points $p$
 on the border, there is a boundary parameter in an open neighbourhood of every
 boundary point such that the local expression $Q(z)dz^2$
 satisfies (\ref{eq:trajectory_condition})
 for some parametrization $\gamma(t)$ of a portion of the real axis containing
 $z(p)$.  We will also say that $R$ is admissible for the quadratic differential
 in this case.

 \begin{definition} \label{de:admissible}
  Let $R$ be a bordered Riemann surface and $\alpha$ be a quadratic differential on $R$.
  We say that $\alpha$ is admissible for $R$ if all but finitely many points $p$ on the boundary
  $\partial R$ are in the image of a trajectory $\Gamma:(a,b) \rightarrow \partial R$ of $\alpha$.
 \end{definition}

 \begin{definition}
  Let $\alpha$ be a quadratic differential on a Riemann surface
   $S$ and $f:R \rightarrow S$ be a local biholomorphism.
  The ``pull-back'' of $\alpha$ to $R$ under $f$ is defined by,
  for a local representation $Q(w)dw^2$ of $\alpha$ and $w=f(z)$ of $f$,
  \[  f^*(Q(w)dw^2) = Q(f(z)) f'(z)^2 dz^2.  \]
 \end{definition}
 \begin{remark}
  If $R$ and $S$ are bordered Riemann surfaces, and $\alpha$ extends meromorphically to
  a neighbourhood of $\overline{S}$ in the double of $S$, and $f$ is holomorphic
   on $\partial R$, then the pull-back can be extended to the boundary.
 \end{remark}
 \begin{remark} \label{re:pull-back_trajectory}
 It is immediately evident that if $\gamma$ is a trajectory of $f^*(\alpha)$ if and only if $f \circ \gamma$ is
 a trajectory of $\alpha$, since
 \[  Q(f(\gamma(t)) f'(\gamma(t))^2 \left(\frac{d \gamma(t)}{dt} \right)^2 = Q(f \circ \gamma(t)) \left( \frac{d f \circ \gamma}{dt} \right)^2.
    \]
  This fact extends to trajectories on the boundary when $f$ and $Q$ are sufficiently regular.
 \end{remark}

 Using the Schwarz
 reflection principle it is easy to see that if
 $Q(w)dw^2$ is admissible for $R_2$ then the pull-back under a conformal
 bijection is admissible for $R_1$.  Note
 that this is not necessarily true if $f$ is not a bijection.   A conformal bijection
 of $R_1$ to $R_2$ extends to a conformal bijection of the doubles $S_i$
 of $R_i$, $i=1,2$, and an admissible quadratic differential on $R_i$
 extends uniquely to a quadratic differential on the doubles by reflection.
 Thus, the statement that an admissible quadratic differential has zeros and poles on
 the boundary has a conformally invariant meaning.

 In this paper we will be concerned entirely with simply connected Riemann
 surfaces conformally
 equivalent to the disk and their double covers with finitely many branch points.
 If $R$ is a simply connected bordered surface, then $R$ and its double are
 conformally equivalent to $\disk$ and $\sphere$ respectively, and the boundary can be identified
 with $\partial \disk$.
 Of course $R$ is conformally equivalent to any simply connected
 domain $\Omega$ in the plane which is not $\mathbb{C}$.  In either case we can
 represent the quadratic differential globally on $\Omega$ as $Q(z)dz^2$ for the global parameter
 $z$ on $\mathbb{C}$.  If we choose $\Omega = \mathbb{D}$ then $Q(z)$ extends to a
 rational function on $\sphere$.
 \begin{remark} \label{re:outer_domain_regularity}
 According to the above definitions, it makes sense
 to say that a quadratic differential is admissible for $\Omega$ even when $\partial
 \Omega$ is highly irregular.
 \end{remark}
 \begin{remark}
   If $\Omega$ is represented as a planar domain bounded by a piecewise analytic Jordan
 curve, then a quadratic differential $Q(z)dz^2$ is admissible for $\Omega$ if and only
 if the boundary segments $\gamma(t)$ satisfy $Q(\gamma(t)) \gamma'(t)^2 < 0$
 for the local meromorphic extension of $Q(z)$ across the boundary curve.

 However,
  if the boundary is piecewise analytic but not a Jordan curve (e.g. so that for some $p \in \partial \Omega$ there
  is an open neighbourhood $U$ of $p$ such that $U \cap \Omega$ has two disjoint components)
  then $Q(z)dz^2$ might not have a consistent extension from the two ``sides'' of $U \cap \partial \Omega$.
  However, the two sides of $U \cap \partial \Omega$ correspond to distinct analytic arcs in the double.
  Thus this problem is avoided in our formulation above.  This subtlety is one of the chief reasons
  we require the formalism of borders and doubles in this paper.
 \end{remark}

 Finally we will need the following elementary theorem \cite[Theorem 8.1]{Pommerenkebook}.
 \begin{theorem} \label{th:zero_classification} Let $Q(z)dz^2$ be a quadratic differential on an open connected set $U
 \subseteq \mathbb{C}$.
 \begin{enumerate}
  \item If $Q(z_0) \neq 0$, then there exists a neighbourhood $V$ of $z_0$ in $U$ and a biholomophism
  $\phi:V \rightarrow W \subset\mathbb{C}$ such that for $w=\phi(z)$ we have $Q(z)dz^2 = dw^2$ (that is,
  $\phi'(z)^2 = Q(z)$).
  \item If $Q$ has a zero of order $n >0$ at $z_0$, then there exists a neighbourhood $V$ of $z_0$ in $U$ and a biholomophism
  $\phi:V \rightarrow W \subset\mathbb{C}$ such that for $w=\phi(z)$ we have $Q(z)dz^2 = w^n dw^2$ (that is,
  $\phi'(z)^2 \phi(z)^n = Q(z)$).
 \end{enumerate}
 \end{theorem}
 It is also possible to classify the poles, but we will not have need of this. We will refer to points where $Q(z_0)$
 has neither a zero nor a pole as regular points.
 \end{subsection}
\begin{subsection}{Simple domains}  In the next section we will define invariants depending on
 pairs of nested simply-connected domains and a quadratic differential.
 In this subsection, we specify the regularity of the inner domain.  In fact, the functionals
 can be extended to much more irregular domains; however, to avoid lengthening this paper needlessly
 we will not pursue this here.
 By conformal invariance, there is no restriction on the regularity of the outer domain
 (see Remark \ref{re:outer_domain_regularity}).

 From now on, a ``conformal disk'' is a bordered Riemann surface conformally equivalent to the disk.
 The reader can replace ``conformal disk'' with $\mathbb{D} = \{z : |z|<1 \}$ everywhere, if she is willing
 to take on faith that the results are conformally invariant.

 \begin{definition} Let $R$ be a conformal disk, with double $\hat{R}$.
  An open connected set $D \subseteq R$ is called ``simple'' if (1) it is simply connected
  (2) there exists a quadratic differential $\alpha$ on the closure of $D$
  in $\hat{R}$ which is admissible for $D$, and (3) this quadratic differential has no poles on $\partial D$.
 \end{definition}
 This class of domains has the following properties: (1) the functionals are easily defined on it
 without introducing analytic difficulties, (2) it includes the extremal domains, and (3) it is
 dense in the set of all simply-connected proper subsets of $\mathbb{C}$ in a certain sense.

 The conformal disk $R$ is itself simple, as can be seen by identifying $R$ and its double with $\disk$ and $\sphere$
 and considering the quadratic differential $dz^2/z^2$ on $\sphere$.
 For a conformal disk $R$, it is clear that a simple domain $D \subseteq R$ is itself a conformal disk.  The
 condition that $D$ be simple imposes a further condition on the regularity of $\partial D$ {\it as it appears in $R$}.
 In principle, one should always say $D$ is simple {\it with respect to $R$}, although
 for brevity we will usually drop this last phrase.

 The boundary of a simple domain $D$ consists of a finite collection of analytic arcs,
 joined at a finite number of ``vertices''.  Let $\mathfrak{V}$ denote the
 set of zeros of a quadratic differential $\alpha$ on $\partial D$ together with those points
 which are an endpoint of only a single trajectory arc.  The latter
 type of point may be zeros or regular points of $\alpha$.  We call this the set of
 vertices of $D$.
 \begin{theorem} \label{th:boundary_arcs_vertices} Let $D$ be a simple domain in a conformal
 disk $R$.   For a quadratic differential $\alpha$ admissible for $D$, the set of vertices
  $\mathfrak{V}=\{v_1,\ldots v_n\}$ on $\partial D$
  is finite.  The complement $\partial D \backslash \mathfrak{V}$ consists of finitely many analytic
  arcs.  Furthermore, at each vertex, there is a neighbourhood which intersects only finitely many of these arcs,
  which meet at the vertex at equally spaced angles.
 \end{theorem}
 \begin{proof}  The fact that there are finitely many vertices follows from the fact that
 the quadratic differential is meromorphic on the compact closure of $R$,
 the compactness of $\partial D$, and the fact that the quadratic differential
  is not identically zero.  The second claim follows
 from applying part (1) of Theorem \ref{th:zero_classification} in local coordinates,
 after observing that at any regular point the trajectory of $\alpha$ through a regular point in a sufficiently
 small neighbourhood is the image of an interval of the imaginary axis under $\phi^{-1}$.  The final claim
 follows from part (2) of Theorem \ref{th:zero_classification}.
 \end{proof}
 \begin{remark}  It is possible that some of the boundary curves of $D$ are arcs of $\partial R$.  These
  are analytic curves in the double of $R$.  If $R$ is identified with $\disk$, then these are arcs of
  $\partial \disk$.
 \end{remark}

 It is elementary that simple domains are dense in the set of simply connected proper subsets of $\mathbb{C}$
 in the following sense.
 \begin{proposition} \label{th:density}
  Let $R$ be a bordered Riemann surface conformally equivalent to the disk, and
  let $D$ be any simply connected subset of $R$.
  Let $f:\mathbb{D} \rightarrow D$ be a conformal bijection.
  There is a sequence of holomorphic maps $f_n:\mathbb{D} \rightarrow D$
  such that $f_n(\mathbb{D}) \subset f_{n+1}(\mathbb{D})$ is a conformal bijection
  for all $n$, $f_n \rightarrow f$ uniformly
  on compact sets, and each $D_n$ is a simple domain bounded by a single analytic Jordan curve.
 \end{proposition}
 \begin{proof}  Let $f_n(z)=f((1-1/n)z)$, let $Q(z)dz^2 = dz^2/z^2$ and set
 \[  Q_n(z)dz^2 = \frac{(f_n^{-1})'(z)^2 dz^2}{f_n^{-1}(z)^2}  \]
 on $D_n$.
 Clearly $f_n(\partial \mathbb{D})$ is an analytic Jordan curve, and since $f_n$ has
 a bijective holomorphic extension to a neighbourhood of $\overline{\mathbb{D}}$, by Remark \ref{re:pull-back_trajectory}
 $Q_n(z)dz^2$ meet the conditions of Definition \ref{de:admissible}.
 \end{proof}
\end{subsection}
\end{section}
\begin{section}{Conformal invariants associated to quadratic differentials}
\begin{subsection}{Harmonic pairs on double covers adapted to quadratic differentials}
  In order to define the conformal invariants, we will need a covering on which the
 quadratic differential has a single-valued square root.  We first define such covering
 and then show that it has the desired properties.

 \begin{definition} \label{de:adapted_cover}
 Let $D$ be a simple
 domain in a conformal disk $R$, and let $\alpha$ be a quadratic differential admissible for $D$ (at least
 one exists by definition).
 We say that $\pi:\tilde{D} \rightarrow D$ is a cover adapted to $\alpha$ if it is a
 double-sheeted cover of $D$ with a branch point of order two at each odd-order
 zero and pole of $\alpha$.
 \end{definition}
  Recall that there are at most finitely many zeros and poles.
  If there are no poles or zeros of odd order, then $\tilde{D}$ consists of two
  disjoint sheets biholomorphic to $D$.  In that case, any curve in $D$ has a two distinct
  lifts, each lying entirely in one sheet.  If there is at least one odd order pole or
  zero, then a closed
  curve $\gamma$ in $D \backslash \{z_1,\ldots,z_k,p_1,\ldots,p_m \}$ lifts to a closed curve
  in $\tilde{D}$ if
 and only if the sum of the winding numbers of $\gamma$ with respect to the
 points $z_i$, $p_j$ is even.

 The double cover adapted to a quadratic differential is uniquely determined up to
 a conformal map.
 \begin{proposition} \label{pr:cover_uniqueness}  Let $D$ be a simple domain in a conformal disk $R$.
 Let $\tilde{D}_1$ and $\tilde{D}_2$ denote two covers of $D$
 adapted to $\alpha$.  There exists a conformal map $\phi:\tilde{D}_1 \rightarrow \tilde{D}_2$.
 \end{proposition}
 \begin{proof}
  If there are no poles or zeros of odd order, then $\tilde{D}_1$ and $\tilde{D}_2$ both consist of two disjoint
  sheets each of which is biholomorphic to $D$, and the claim follows immediately.

  Now assume that there is at least one zero or pole of odd order.
  Fix $z_0 \in D$ and let $p_i$ and $q_i$ denote the two preimages of
  $z_0$ under $\pi_i$, $i=1,2$.  The map $\pi_1:\tilde{D}_1 \rightarrow D$ has a single-valued lift to a map
  $\hat{\pi}_1: \tilde{D}_1 \rightarrow \tilde{D}_2$ such that $\hat{\pi}_1(p_1)=p_2$ (by lift, we mean that, $\pi_2 \circ \hat{\pi}_1 = \pi_1$).
  To see this it is enough to show that $\pi$ induces a map from the fundamental group of the covering
  $\tilde{D}_1$ into that of $\tilde{D}_2$.
  To this end observe that
  every non-trivial element of the fundamental group at $z_0$ of the covering $\tilde{D}_1$
  can be represented as a lift of a closed curve $\gamma$ such that the sum of the winding numbers around the odd zeros and poles of $Q$ is even.  However, this is
  precisely the condition that $\gamma$ be a representative of the covering group
  of $\tilde{D}_2$.  This proves the claim.
  Similarly, $\pi_2$ has a lift to a map $\hat{\pi}_2:\tilde{D}_2 \rightarrow \tilde{D}_1$ such that $\hat{\pi}_2(p_2)=p_1$.  Both maps $\hat{\pi}_i$
  are holomorphic.

  Since $\pi_2 \circ \hat{\pi}_1 = \pi_1$, we have that $\pi_2 \circ \left( \hat{\pi}_1 \circ \hat{\pi}_2 \right) = \pi_1 \circ \hat{\pi}_2 = \pi_2$.
  Thus $\hat{\pi}_1 \circ \hat{\pi}_2$ is a lift of the identity; since $\hat{\pi}_1 \circ \hat{\pi}_2(p_2)=p_2$ by uniqueness of lifts it must be the identity.  Similarly $\hat{\pi}_2 \circ \hat{\pi}_1$ is the identity.  Setting $\phi=\hat{\pi}_1$ we have proven the proposition.
 \end{proof}

%
 The condition that a quadratic differential be admissible for $D$ implies that
 the primitive of its square root on the double cover has a single-valued real part, at least in
  a doubly-connected domain near the boundary $\partial D$.
  The next proposition formulates this precisely.

  Given a quadratic differential $\alpha$ admissible for $D$ and a double
  cover $\pi:\tilde{D} \rightarrow D$ adapated to $\alpha$, since $\pi$ is a local biholomorphism away from branch points,
  $\pi^*\alpha$ is a well-defined quadratic differential on $R$ minus the branch points.
  It is easy to see that $\pi^*\alpha$ remains bounded at
  branches and therefore extends to a quadratic differential on $\tilde{D}$.
 \begin{theorem} \label{th:primitive_is_constant} Let $D$ be a conformal disk and let $\alpha$ be an admissible
 quadratic differential for $D$.
  Let $\tilde{D}$ be a double-cover of $D$ adapted to $\alpha$.
   Let $F:\disk \rightarrow D$ be a conformal bijection. Let $\mathfrak{B}$ be the set of branch points in $\tilde{D}$
    of the cover $\pi:\tilde{D} \rightarrow D$.
   \begin{enumerate}
   \item
    There is a well-defined meromorphic one-form $\beta$ on $\tilde{D}$
     such that $\beta^2 = \pi^*(\alpha)$.
  \item   The one-form
     $\alpha$ has a multi-valued holomorphic primitive $x$ on $\tilde{D} \backslash \mathfrak{B}$.
      For some $0<r<1$, $x$ has
      a single-valued real part on $\pi^{-1} F(r<|z|<1)$.
       The real part of the primitive $x$ extends to a well-defined harmonic function on any
  domain of the form $D \backslash \Omega$ where $\Omega$ is a simply connected
  domain, containing the odd order poles of $\alpha$ whose closure
  is in $D$.,
  \item $q = \text{Re}(x)$ extends continuously to a constant function on $\partial \tilde{D}$.
  \end{enumerate}
 \end{theorem}
 \begin{remark} \label{re:double_cover_has_border}  If $D$ is a conformal disk (in fact
 any bordered Riemann surface),
  it is easily seen that if $\tilde{D}$ is a double cover of $D$ all of
  whose branch points are in the interior, then $\tilde{D}$ is also a bordered Riemann
  surface.  For any border chart $\phi$ in a neighbourhood of $p \in \tilde{D}$,
  $\phi \circ \pi$ is a border chart near each of the two points in
  $\pi^{-1}(p)$ for locally biholomorphic choices of $\pi^{-1}$.   We will use this in the following
  proof.
 \end{remark}
 \begin{proof} Using pull-backs, we may assume that $D$ is the unit disk $\disk$ and
  therefore $\alpha = Q(z)dz^2$ in terms of the global parameter $z$.  (In fact, by admissibility
  $Q(z)$ must be a rational function on $\sphere$).

  To prove (1), first we observe that at any $p \in \tilde{D} \backslash \mathfrak{B}$,
  there are two square roots of $\pi^*Q(z)dz^2$ defined in a neighbourhood of $p$
  as follows.  Assume first that $\pi(p)$ is not a pole or zero of $Q(z)dz^2$.
  If $\zeta$ is a local holomorphic coordinate in a neighbourhood of $p$
  and $\pi^* Q(z)dz^2 = h(\zeta) d\zeta^2$ we can locally set $\beta = \sqrt{h(\zeta)} d\zeta$,
  and there are precisely two choices of $\sqrt{h(\zeta)}$.   Using the
  transformation property of quadratic differentials, it is easily verified that
  this pair of locally defined one forms is independent of the choice of local
  coordinates.  If now $\pi(p)$ is a zero or pole of $Q(z)dz^2$, then it must
  have even order, since $p \notin \mathfrak{B}$.  Since $\pi$ is a covering
  (and therefore has non-zero derivative in local coordinates), $p$ is also a
  zero or pole of $\pi^*Q(z)dz^2$ of the same order.  In local
  coordinates $\zeta$ we have that $\pi^* Q(z)dz^2 = (\zeta - \zeta(p))^{2n} H(\zeta) d\zeta^2$
  for some non-vanishing holomorphic function $H$ and we again have two square roots
  $(\zeta - \zeta(p))^n \sqrt{H(\zeta)}$ in a neighbourhood of $p$.  Note
  that on a sufficiently small
  open set $U$ containing $p$,  the locally defined one-form $\beta$ is $\pi^*(\delta)$ for some
   one-form $\delta$ on $\pi(U)$.

  Now assume that $\gamma:[0,1] \rightarrow D$ is a closed curve winding once around a odd
  order zero or pole, say $p$, but winding around no other odd order zero or pole.  The lift of $\gamma$ is such that
   $\gamma(0)$ and $\gamma(1)$ lie on different sheets.
   If we continue a choice of square root of $\pi^*(Q(z)dz^2)$
  along $\gamma$, the corresponding one-form $\delta$ defined on a neighbourhood of
  $\gamma([0,1])$ on $D$ picks up a sign change each time it winds around $p$.  Thus in general the sign
  change of the continuation of $\delta$ along any closed curve $\gamma$ is $(-1)^k$
  where $k$ is the sum of the winding numbers of $\gamma$ around the odd order
  zeros and poles.  The sign change of the corresponding continuation $\beta = \pi^*\delta$
  is the same.

  Now fixing a point $p_0 \in \tilde{D} \backslash \mathfrak{B}$ we make a choice of
  square root in a neighbourhood, and analytically continue it to $\tilde{D} \backslash
  \mathfrak{B}$.  It needs to be shown that this continuation is single valued.
  Let $\Gamma$ be a closed curve in $\tilde{D}$ through $p_0$.  Since the sum of
  the winding numbers of $\pi \circ \Gamma$ around odd order zeros and poles is
  even, the continuation of $\beta$ along $\Gamma$ is single-valued
  by the previous paragraph.  Thus $\beta$ is well-defined
  on $\tilde{D} \backslash \mathfrak{B}$.

   To see that $\beta$ extends
  meromorphically to a branch point $b \in \tilde{D}$, observe that there is a local coordinate $z$ say
  in a neighbourhood of $b$, and a local coordinate $w$ in a neighbourhood of $\pi(b)$,
   in which $\pi$ has the form $w=z^2$ (with $w=0$ and $z=0$ corresponding to $\pi(b)$ and
   $b$ respectively).  If $\alpha$ in local coordinates
   has the form $Q(w)dw^2$ then $\pi^*\alpha$ has the form $4 Q(z^2) z^2 dz^2$, so
   $\pi^*\alpha = z^{2m} h(z)dz^2$ for some integer $m$ and holomorphic $h$ such that
   $h(0) \neq 0$.  Thus $\pi^* \alpha$ has a well-defined square root $z^m \sqrt{h(z)}dz$ in a neighbourhood of
   $b$, and this must agree with $\beta$ on the punctured neighbourhood.   This proves (1).

  We prove (2) and (3) simultaneously.  Let $r$ be large enough that $\pi^{-1}(F(r<|z|<1))$
  contains no zeros and poles of $Q(z)dz^2$.  On this domain define the (generally multi-valued)
  analytic function
  \[    x(\zeta)  = \int_{\zeta_0}^\zeta \beta.          \]
  By Remark \ref{re:double_cover_has_border}, $\pi^*(Q(z)dz^2)$ continues analytically to $\partial \tilde{D}$,
  and hence so do $\beta$ and $x$.

  We now show that $\text{Re} x$ is constant on $\partial \tilde{D}$.
  Let $\gamma(t)$ parameterize a portion of the boundary curve
  of $\partial \tilde{D}$.  We have that
  \[  Q(\pi \circ \gamma(t)) \left( \frac{d (\pi \circ  \gamma)}{dt}(t) \right)^2 \leq 0 \]
  and thus for either choice of square root of $Q$ in a neighbourhood of $\gamma(t)$ we have
  \[   \mbox{Re} \left(\sqrt{Q(\pi \circ \gamma(t))} \cdot \frac{d (\pi \circ \gamma)}{dt}(t)
  \right) = 0 \]
  and therefore $\pi^*(Q(z)dz^2)$ has the same property with respect to $\gamma$.
  Thus the one-form $\beta$ evaluated in a direction tangent to $\partial \tilde{D}$
  is pure imaginary.  In particular, $\text{Re}(x)$ is constant on $\partial \tilde{D}$.

  It is clear that $x$ extends analytically to any domain of the specified type, although
  of course it might be multi-valued.  It remains to prove that the real part of
  $x$ is single valued.  Now $\beta$
  (which equals $x'(\zeta) d\zeta$ in coordinates $\zeta$)
  is a single-valued holomorphic one-form on the doubly-connected
  domain $\pi^{-1}(F(r<|z|<1))$, and by the previous paragraph
  the period of $\beta$ on this domain is pure imaginary.  This proves that
  $\text{Re}(x)$ is single valued on this domain.
 \end{proof}
 \begin{remark}  It is immediately seen that $\beta$ extends to a meromorphic differential
  on the double of $\tilde{D}$.
 \end{remark}

 Let $\partial$ denote the differential operator given in local coordinates by
 \[  \partial h = \frac{\partial h}{\partial z} dz.   \]
  In the future we will sometimes denote the one-form $\beta$ by $\partial x$ or
  locally by $x'(\zeta) d\zeta$.  Note that because $\text{Re}(x)$
  may have singularities, the maximum principle cannot be applied so it need not be constant on $D$.
  Clearly $x$ is determined uniquely up to an additive
 constant.

  In the literature, $x \circ \pi^{-1}$ is called
  the canonical or straightening map of $Q(z)dz^2$.  The terms can
  refer to either a single-valued
  choice of $x \circ \pi^{-1}$ on the domain $D$ minus branch cuts or the multi-valued
  function.
 In the proof of Theorem \ref{th:primitive_is_constant} it appeared that
 if a curve is a trajectory of a quadratic differential  then the corresponding function $\text{Re} (x)$ is constant
 on its lift.  The converse is also true.  In a local parameter $z$, observe that $x$ and $Q(z)dz^2$ are related by
 \begin{equation} \label{eq:Q_and_x_relation}
  Q(z)dz^2 = \left( \frac{\partial}{\partial z} x \circ \pi^{-1}(z) \right)^2 dz^2.
 \end{equation}
 If $\alpha$ denotes the quadratic differential, then we can globally write
  $\alpha = \left( \partial (x \circ \pi^{-1}) \right)^2$.
 \begin{proposition} \label{pr:trajectories_and_constants}  Let $D$ be a conformal disk and $\alpha$ be a quadratic differential
  admissible for $D$.  Let $\tilde{D}$ be a double cover adapated to $\alpha$ with branch points $\mathfrak{B} \subseteq \tilde{D}$
   and let $x$ be the multi-valued meromorphic function on $\tilde{D} \backslash \mathfrak{B}$  of Theorem \ref{th:primitive_is_constant};
   that is $\alpha = (\partial (x \circ \pi^{-1}))^2$.
   A curve $\gamma(t)$ is a trajectory of $\alpha$ if and only if $\text{Re}(x)$ is constant
   on $\pi^{-1} \circ \gamma$ for any local choice of $\pi^{-1}$.
 \end{proposition}
 \begin{proof}  By conformal invariance we can assume that $D$ is the disk $\disk$, and
 thus we have a global parameter $z$ with $\alpha = Q(z)dz^2$.
  The first claim follows directly from the proof of Proposition \ref{pr:trajectories_and_constants}.
  On the other hand, if $\text{Re}(x)$ is constant on $\gamma(t)$ then for some local
  determination of $\pi^{-1}$, setting $\Gamma = \pi \circ \gamma$ we have
  \begin{align*}
    0 & = \frac{d}{dt} \text{Re} \left[x(\gamma(t)) \right] =
     \frac{d}{dt} \text{Re} \left[ x \circ \pi^{-1} \circ \Gamma(t) \right]
    \\
    & = \text{Re} \left[ \frac{\partial x \circ \pi^{-1}}{\partial z} \circ \Gamma \cdot
    \frac{d \Gamma}{dt} \right]
  \end{align*}
  Thus
  \[  Q(\Gamma(t)) \frac{ d\Gamma}{dt}^2 = \left[ \frac{\partial x \circ \pi^{-1}}{\partial z} \circ \Gamma \cdot
    \frac{d \Gamma}{dt} \right]^2  \leq 0.  \]
 \end{proof}

 Let $D_1$ and $D_2$
 be simple domains such that $D_2 \subseteq D_1$, and let  $Q(z)dz^2$ be an
 admissible quadratic differential for $D_1$.  Assume that all poles of $Q(z)dz^2$
 which are contained in the interior of $D_1$ are also contained in the interior of $D_2$.
 Let $\pi:\tilde{D}_1 \rightarrow D_1$ be a double cover of $D_1$ adapted to $Q(z)dz^2$.
  Let $x_1$ be the function guaranteed by Theorem \ref{th:primitive_is_constant}
  and let $q_1(z) = \mbox{Re}(x(z))$.  Assume that the additive constant of $x_1$ is chosen
  so that $q_1=0$ on $\partial \tilde{D}_1$.
 \begin{definition} \label{de:induced_harmonic}  Let $D_1$ be a conformal disk and let $D_2$ be a simple
 subdomain of $D_1$.  Let $\alpha$
  be a quadratic differential admissible for $D_1$ all of whose poles are in $D_2$.  Let $\tilde{D}_1$
  be a double cover adapted to $\alpha$ and  $\tilde{D}_2 = \pi^{-1}(D_2)$.  Let $x_1$ be one of
  the primitives of the square root of $\alpha$ on $\tilde{D}_1$ and set $q_1 = \text{Re}(x_1)$,
  with the additive constant chosen so that $q_1 = 0$ on $\partial \tilde{D}_1$.
   Let $u$ be the unique harmonic function on $\tilde{D}_2$ such that $u=q_1$ on $\partial \tilde{D}_2$
  and set $q_2 = q_1 - u$.
  \begin{enumerate}
    \item  We call $\pm(q_1,q_2)$ the harmonic pair induced by $(D_1,D_2,Q(z)dz^2)$.
    \item We call $\alpha_2 = \left( \partial (x_2 \circ \pi^{-1}) \right)^2$ the
    quadratic differential on $D_2$ induced by $\alpha$.  In a local coordinate $z$, we have
    that $\alpha_2$ can be written
    $Q_2(z)dz^2 = 4 \left(\frac{\partial q_2 \circ \pi^{-1}(z)}{\partial z} \right)^2dz^2$.
  \end{enumerate}
 \end{definition}
 \begin{remark}[convention for disconnected cover] \label{re:harmonic_pair_convention} In the case that $Q(z)dz^2$ has no double poles
  or zeros, the double cover $\tilde{D}_1$ has two connected components, as observed above.  In
  this case, we adopt the following convention.  The primitive $x_1$ is chosen so that for any
  fixed point $z \in D_1$, the two values of $q_1 = \mathrm{Re} x_1$ at $\pi^{-1}(z)$ differ by a sign.
  With this restriction, there are two (rather than four) possible choices of harmonic pair $\pm (q_1,q_2)$ on $\tilde{D}_1$, in agreement with the case that $\tilde{D}_1$ is connected.
 \end{remark}
 \begin{remark}  Applying Theorem \ref{th:primitive_is_constant}, we see that
  $q_2$ is single-valued near $\partial \tilde{D}_2$ in the sense of the Proposition.
  The function $u$ is single-valued and non-singular on $\tilde{D}_2$.
 \end{remark}

 Of course it must be verified that
 the induced harmonic pair and quadratic differential are well-defined.
 Observe that $q_2$ is uniquely determined by $x_1$ and $D_2$.  Furthermore, $x_1$ (and hence $q_1)$ is determined up
 to a sign; that is given one such function $q_1$, $\pm q_1$ are the only two functions satisfying the definition,
 and furthermore
 for the unique non-trivial deck transformation $g:\tilde{D}_1 \rightarrow \tilde{D}_1$
 we have that $q_1(g(z))=-q_1(z)$.  Clearly if $q_2$ is the harmonic function on $\tilde{D}_2$
 associated with $q_1$ as in Definition \ref{de:induced_harmonic}, then $-q_2$ is the harmonic
 function on $\tilde{D}_2$ associated with $-q_1$.  Thus $(q_1,q_2)$ and $(-q_1,-q_2)$ are the
 only pairs satisfying Definition \ref{de:induced_harmonic}.
 Thus the harmonic pair $\pm(q_1,q_2)$ is well-defined.

 Furthermore we have the following.
 \begin{proposition}  \label{pr:induced_qd_wd_and_admissible}
  Let $D_1$ be a conformal disk and $D_2$ be simple domain in $D_1$.
  Let $\alpha_1$ be an admissible quadratic differential for $D_1$.  Assume that all
  poles of $\alpha_1$ are contained in the interior of $D_2$.
  \begin{enumerate}
  \item
  The quadratic differential
  $\alpha_2$ on $D_2$ induced by $\alpha_1$ is well-defined.  Furthermore, $\alpha_2$
  is admissible for $D_2$.
  \item If $\pm(q_1,q_2)$ is the harmonic pair associated with $(D_1,D_2,\alpha_1)$ then
   $q_2$ is the singular harmonic function on $\tilde{D}_2$ associated to $\alpha_2$ as in Theorem
   \ref{th:primitive_is_constant}.
 \end{enumerate}
 \end{proposition}
 \begin{proof}
  Let $g$ be the deck transformation of $\tilde{D}_1$.
  We claim that $q_2(g(z)) = - q_2(z)$.  To see this, let $u$ be the solution to the Dirichlet
  problem on $\partial \tilde{D}_2$ with boundary values equal to $q_1$ on $\partial \tilde{D}_2$,
  and recall that $q_2 = q_1 - u$.  We have that for all $z \in \partial \tilde{D}_2$,
  $u(g(z))= q_1(g(z))= - q_1(z) = - u(z)$.  Thus $u(g(z)) + u(z)$ is zero on $\partial \tilde{D}_2$
  and so $u(g(z)) = - u(z)$ for all $z \in \tilde{D}_2$.  In particular $q_2(g(z))= - q_2(z)$
  for all $z \in \tilde{D}_2$.

  Now fix an open set $V \subseteq D_2$, with a local coordinate $z$, where $V$
  is chosen so that $\pi^{-1}(V)$ has precisely
  two disjoint components $U$ and $\hat{U}$
  and $\pi$ has biholomorphic inverses $\pi^{-1}$ and $\hat{\pi}^{-1}$
  on $V$.  We thus have that $\pi^{-1} = g \circ \hat{\pi}^{-1}$ on $V$.  By the previous paragraph,
  \[  \frac{\partial q_2 \circ \pi^{-1}}{\partial z}(z)dz
     = - \frac{\partial q_2 \circ \hat{\pi}^{-1}}{\partial z}(z)dz  \]
  for all $z \in U$.  Thus in local coordinates the expression
  \[  Q_2(z)dz^2 =  4 \left(\frac{\partial q_2 \circ \pi^{-1}}{\partial z}(z) \right)^2dz^2 \]
  for $\alpha_2$
  is independent of the local choice of $\pi^{-1}$.  This shows that the quadratic differential
  on $D_2$ induced by $\alpha$ is well-defined.

  Thus to show that $\alpha_2$ is admissible for $D_2$
  we need only show that the boundary is a trajectory.  Let $x_2$ be the multi-valued meromorphic
  function on $\tilde{D}_2 \backslash \mathfrak{B}$ whose real part is $q_2$.  By the
  Cauchy-Riemann equations, in local coordinates we have
  \[  Q_2(z)dz^2 = 4 \left( \frac{\partial q_2 \circ \pi^{-1} }{\partial z} \right)^2 dz^2
     = \left(  \frac{\partial x_2 \circ \pi^{-1}}{\partial z} \right)^2 dz^2;   \]
  That is, $\alpha_2 = \left( \partial (x_2  \circ \pi^{-1} ) \right)^2$.
  Thus $\alpha_2$ is admissible by Proposition \ref{pr:trajectories_and_constants}.
  This also proves the second claim.
 \end{proof}
 \begin{proposition} \label{pr:induced_properties}
  Let $D_1$ be a conformal disk and $D_2$ be a simple domain in $D_1$.  Let $\alpha_1$ be a quadratic
  differential admissible for $D_1$.
  \begin{enumerate}
   \item If $\tilde{D}_1$ and $\hat{D}_1$ are two distinct covers of $D_1$ adapted to $\alpha_1$,
    and $g:\tilde{D}_1 \rightarrow \hat{D}_1$ is a conformal map,
    then the corresponding multi-valued holomorphic functions $x_1$ and $\hat{x}$ satisfy $\hat{x}_1 \circ g = x_1$.
    Similarly $\hat{q}_i \circ g = q_i$ for $i=1,2$.
   \item  The quadratic differential on $D_2$ induced by $\alpha_1$ is independent of the choice
   of cover adapted to $\alpha_1$.  In particular, the induced quadratic differential is well-defined.
  \end{enumerate}
 \end{proposition}
 \begin{proof}  The first claim follows from Proposition \ref{pr:cover_uniqueness}.
  The second is immediate.
 \end{proof}
 It is also elementary that
 \begin{proposition}  If $D_1$ is a conformal disk, $\alpha_1$ is an
 admissible quadratic differential for $D_1$, and $\tilde{D}_1$ is a cover adapted
 to $\alpha_1$, then the harmonic pair induced by $(D_1,D_1,\alpha_1)$
  is $\pm(q_1,q_1)$.   Furthermore the induced differential $Q_2(z)dz^2$ on $D_2$ equals $Q(z)dz^2$.
 \end{proposition}

 Finally, the harmonic pairs and induced quadratic differential also have a kind of transitivity property.
 \begin{proposition}  \label{pr:induced_things_are_modules} Let $D_1$ be a conformal
 disk and let $D_2$ and $D_3$ be simple domains in $D_1$ satisfying $D_3 \subseteq D_2 \subseteq D_1$.
  Let $\alpha_1$ be an admissible quadratic differential for $D_1$, all of
  whose poles are in $D_3$.  Let $\pi_1:\tilde{D}_1 \rightarrow D_1$
  be a double cover adapted to $Q_1(z)dz^2$.  For $i=1,2$ let
  $\alpha_i$ be the quadratic differentials on $D_i$ induced by $\alpha_1$.
  \begin{enumerate}
   \item $\left. \pi_1 \right|_{\tilde{D}_2}$
    is a double cover of $D_2$ adapted to $\alpha_2$.
   \item  If $\pm(q_1,q_i)$
    are the harmonic pairs induced by $(D_1,D_i,\alpha_1)$ for $i=1,2$ then
    $\pm(q_2,q_3)$ is the harmonic pair induced by $(D_2,D_3,\alpha_2)$.
   \item $\alpha_3$ is the
    quadratic differential on $D_3$ induced by $\alpha_2$.
  \end{enumerate}
 \end{proposition}
 \begin{proof}
  We can assume that $D_1 = \disk$, so that we have
  a global parameter $z$.  Let $\alpha_i = Q_i(z)dz^2$ for $i=1,2,3$.
  We first prove (1).
  It is immediate that $x_2$ has the same poles as $x_1$, and of the same order, since $q_1 - q_2$ is non-singular.
  Thus $Q_1(z)dz^2$ and $Q_2(z)dz^2$ have odd order poles at precisely the same points.

  Since $Q_2$ is a perfect square at all points which are not branch points, $Q_2$ cannot
  have an odd order zero unless $Q_1$.  So
  we need only show that if $Q_1(z)dz^2$ has an odd order zero at a point $w_0 \in D_2$, then
  $Q_2(z)dz^2$ also has an odd order zero at $w_0$.  Let $g$ denote the unique deck
  transformation of
  order two on $\tilde{D}_1$.  Assume that $z_0 \in D_2 \subseteq D_1$ is an odd order zero of $Q_1(z)dz^2$.
  Thus $w_0=\pi_1^{-1}(z_0)$ is a branch point of $\pi_1$ and hence a fixed point of $g$.  Since
  $q_2(g(w))=-q_2(w)$
  for all $w$ we also have that in a local coordinate $w$
  \begin{equation} \label{eq:temp_flip}
   \frac{\partial q_2}{\partial w}(w_0)dw =
   \frac{\partial q_2}{\partial w}(g(w_0))dw =-\frac{\partial q_2}{\partial w}(w_0) dw
  \end{equation}
  so $w_0$ is a zero of $\partial q_2/\partial z$ and hence $z_0=\pi_1(w_0)$ is a zero of
  \begin{equation} \label{eq:Q2temp}
   Q_2(z) dz^2 = 4 \left( \frac{\partial q_2 \circ \pi_1^{-1}}{\partial z}(z) \right)^2 dz^2.
  \end{equation}
  This proves (1).

  Next we prove (2).  By Proposition \ref{pr:induced_qd_wd_and_admissible} $q_2$ is the
  singular harmonic function associated to $Q_2(z)dz^2$ as in Theorem \ref{th:primitive_is_constant}.
  Since $q_3=0$ on $\partial \tilde{D}_3$ and $q_2-q_3=q_2 - q_1 + q_1 - q_3$ is harmonic on $\tilde{D}_3$,
  it follows that $\pm(q_2,q_3)$ is the harmonic pair induced by $(D_2,D_3,Q_2(z) dz^2)$.

  (3) is an immediate consequence of (2).
 \end{proof}
\end{subsection}
\begin{subsection}{Definition of the conformal invariants}

  Given a smooth function $h$ on a Riemann surface $R$, we define a differential
  operator $\ast dh$ as follows.   On any domain $U \subset
 R$ with local parameter $z=x+iy$ and real-valued function $h$ define the one-form
 \[  \ast dh = \frac{\partial h}{\partial x} dy - \frac{\partial h}{\partial y}dx.  \]
 $\ast$ is often called the Hodge star operator.

   The expression
 \begin{equation} \label{eq:star_dh_using_dbydz}
  \ast dh = \mbox{Re}\left( \frac{2}{i} \frac{\partial h}{\partial z} dz \right)
 \end{equation}
  is often convenient.
 For example, it is easily computed that
 if $z=g(w)$ is a conformal map defined on an open neighbourhood of a contour
 $\gamma$ then for any real maps $h_1$ and $h_2$
 \begin{eqnarray}  \label{eq:star_change_o_coord}
  \int_\gamma h_1 \circ g \ast d \left( h_2 \circ g\right)
   &=&  \mbox{Re} \left( \int_\gamma h_1 \circ g  \frac{2}{i} \frac{\partial h_2}{\partial z} \circ g \cdot g'(w) \,dw \right) \\ \nonumber
  & = & \mbox{Re} \left( \int_{g \circ \gamma} h_2 \frac{2}{i} \frac{\partial h_2}{\partial z} \,dz \right) \\ \nonumber
 & = & \int_{g \circ \gamma} h_1 \ast d h_2.
 \end{eqnarray}
 This amounts to the same thing as the fact that the Hodge star operator is independent of the choice of local parameter $z$.

 Finally observe that in a local parameter $z = x+iy$,
 \[  d \ast dh = \triangle h \cdot dx \wedge dy = -2 i \, \partial \overline{\partial} h \cdot d\bar{z} \wedge dz.  \]
 where $\partial$ and $\overline{\partial}$ are the standard d- and d-bar operators, and
 \[  \triangle h = \frac{\partial^2 h }{\partial x^2}  +  \frac{\partial^2 h }{\partial y^2}  \]
 is the Laplacian.  It is easily verified that $d \ast dh$ is independent of the choice of local parameter.
 \begin{remark}
  Note that if $\gamma$ is a positively oriented curve in a local coordinate $U$, $n$ denotes the  normal
   directed to the right of travel and $ds$ denotes infinitesimal Euclidean arc length, then
  \begin{equation} \label{eq:nehari_hodge_star}
   \int_{\gamma} \ast dh = \int_{\gamma} \frac{\partial h}{\partial n} ds .
  \end{equation}
  This is the notation employed by Nehari.  We will use this notation ahead when computing examples in the plane.
 \end{remark}

 We may now define the conformal invariants.
 \begin{definition}[Module associated to $(D,D_1,\alpha)$] \label{de:general_module}  Let $D_1$ be a conformal disk and $D_2 \subseteq D_1$ be a simple
  domain in $D_1$.  Let $\alpha$ be a quadratic differential which is
  admissible for $D_1$.
 Let $\pi:\tilde{D}_1 \rightarrow D_1$ be a double cover of $D_1$ adapted to $Q(z)dz^2$.
  Let $\pm(q_1,q_2)$ be the harmonic pair induced by $(D_1,D_2,\alpha)$ on $\tilde{D}_1$.  We call
  \begin{equation} \label{eq:def_module}
    m(D_1,D_2,\alpha) = \int_{\partial \tilde{D}_2} q_1 \ast dq_2
  \end{equation}
  the module of $(D_1,D_2,\alpha)$.
 \end{definition}
 The meaning of this integral requires clarification.   The
 issues are as follows.
 If $D_2$ is bordered by a single analytic curve, then this is an ordinary contour integral.  Consider however the following
 example.  Let
 $D_2$ be the disk with a single radial slit, so that $\tilde{D}_2$ also possesses at least one slit $\Gamma$
 say.
 It is clear that $q_2$ will extend continuously
 to $\Gamma$ (in fact all of $\partial \tilde{D}_2$ in general).  However $\ast d q_2$ will have two distinct extensions,
 one for each ``side'' of the slit.  It is natural in
 this case to interpret the integral along the boundary of $\partial \tilde{D}_2$ as containing one integral
 for each extension of $\ast dq_2$, with opposite orientations.  We will make
 this precise below.

 It also needs to be established that this module is well-defined.  There are two issues:
 first, that the integral converges, and second, that the
 module depends only on $D_1$, $D_2$ and $\alpha$ as the notation suggests, and not on
 the choice of double cover.  Although the convergence is elementary, there are many details
 to address, so we have relegated the proof to an appendix in order not to interrupt
 the flow of the paper.
 The remainder of this section will be devoted to establishing that the
 integral is well defined,
 and also that the module is conformally invariant.

 We now clarify
 the meaning of the integral (\ref{eq:def_module}).
 The function $\ast dq_2$ has a one-sided extension in
 the following sense.  For any point $p \in \partial \tilde{D}_2$ such that $\pi(p)$ is not
 a zero of $\alpha$, by \ref{th:zero_classification}
 there is an open disk $B$ centred on $p$ such that $B \backslash \partial \tilde{D}_2$
 consists of two connected components $U$ and $V$, each of which is either in $\tilde{D}_2$ or
 disjoint from it.  At least one of $U$ or $V$ must be contained in $\tilde{D}_2$;
 assume that it is $U$.  Since $\ast dq_2$ is a harmonic
 form and $\partial \tilde{D}_2 \cap B$ is an analytic curve (it is in particular
 locally of the form $\mathrm{Re}(-2i\partial h/\partial z \,dz)$ for some holomorphic $h$ by (\ref{eq:star_dh_using_dbydz}),
 it has a harmonic extension
 to some open neighbourhood of $U \cup \left( \partial \tilde{D}_2 \cap B \right)$.  If $V$ is
 also in $\tilde{D}_2$, the same argument applies to $V$; however, the two extensions do
 not in general agree on their overlap.
 We abbreviate the above paragraph by saying that $\ast dq_2$ has a harmonic extension to $\partial \tilde{D}_2$ from
 one or two sides
 in a neighbourhood of any point $p$ such that $\pi(p)$ is not a zero.

 Let $\mathfrak{V} \subset \partial D_2$ consist of all points in $\partial D_2$ which are either
 zeros of $\alpha$ or terminal points of trajectories of $\alpha$ (which may in fact be regular points).
 We call this the set of vertices of $\partial D_2$.  We will also call the set $\pi^{-1}(\mathfrak{V})$
 the vertices of $\partial \tilde{D}_2$.

 Let $\Gamma$ be an arc of $\partial \tilde{D}_2$ whose endpoints are vertices,
 and containing no other vertices.   We adopt the convention that the endpoints are not
 in $\Gamma$ (or any analytic arcs below).
 By the argument above, one of two possibilities hold: either (1) for all $p \in \Gamma$,
 there is an open set $B$ containing $p$ such that $B \cap \tilde{D}_2$ has precisely two connected components
 and precisely one of these components is in $\tilde{D}_2$, or (2) for all $p \in \Gamma$, there
 is an open set $B$ containing $p$ such that $B \cap \tilde{D}_2$ has precisely two connected components
 and both are in $\tilde{D}_2$.   If $\Gamma$ satisfies the first condition we call it
  a one-sided boundary arc and if it satisfies the second condition we call it a two-sided boundary arc.
 The same reasoning and terminology holds for sub-arcs of $\partial D_2$ with this property.

\begin{definition} Let $D_2$ be a simple subset of a conformal disk $D_1$.  A complete set of
 maximal boundary arcs of $\tilde{D}_2$ is a collection
 $\gamma_i$, $i=1,\ldots,2m$ oriented analytic arcs with the following properties:
 \begin{enumerate}
  \item $\cup_{i=1,\ldots,m} \gamma_i \cup \pi^{-1}(\mathfrak{V}) = \partial \tilde{D}_2$
  \item the image of $\gamma_i$ joins two vertices, and contains no other vertices
  \item for every one-sided arc $\Gamma$ of $\partial \tilde{D}_2$, there is precisely
  one $\gamma_i$ with the same image; $\gamma_i$ is positively oriented with respect to $\tilde{D}_2$
  \item for every two-sided arc $\Gamma$ of $\partial \tilde{D}_2$, there are precisely two
   curves $\gamma_i$ and $\gamma_j$ with the same image as $\Gamma$; $\gamma_i$ and $\gamma_j$ have opposite orientations.
 \end{enumerate}
 \end{definition}
 It is clear that the collection of $\gamma_i$ are determined uniquely up to ordering.
 Also, given a point on a two-sided arc $\gamma_i$ and a neighbourhood $B$ of $p$ such
 that $B \cap \tilde{D}_2$ has two connected components, $\gamma_i$ will be positively
 oriented with respect to precisely one of the components
 of $B \cap \tilde{D}_2$.  We extend $\ast dq_2$ from this component to $\gamma_i$.
 With this convention we define the integral (\ref{eq:def_module}) as follows.
 \begin{definition} \label{de:integral_with_corners}  Let $D_2$ be a simple domain
  in a conformal disk $D_1$.  Provided that each integral converges, we define
  \[   \int_{\partial \tilde{D}_2} q_1 \ast dq_2 = \sum_{i=1}^m  \int_{\gamma_i} q_1 \ast dq_2  \]
  where $\gamma_1,\ldots,\gamma_m$ are a complete collection of maximal boundary arcs satisfying
  properties (1) - (4) above, and on each arc $\gamma_i$ we choose the harmonic
  extension of $\ast dq_2$ determined by the orientation of $\gamma_i$.
 \end{definition}

 The convergence of this integral is guaranteed by the following theorem.
 The proof is given in the Appendix \ref{se:appendix}.
 \begin{theorem} \label{th:convergence_and_nice_approximation}
  Let $D_2$ be a simple domain in a conformal disk $D_1$.  Let $\alpha$ be a quadratic differential
  which is admissible for $D_1$ and let $\pm (q_1,q_2)$ be the induced harmonic pair, and let $\gamma_1,\ldots,\gamma_m$
  be a complete set of maximal boundary arcs of $\partial \tilde{D}_2$.
    Each integral
  \begin{equation} \label{eq:integral_conv_thm_temp}
    \int_{\gamma_i} q_1 \ast dq_2
  \end{equation}
  converges.  Furthermore, letting $F:\disk \rightarrow D_2$ be a
  conformal bijection and setting $C_r$ to be the curve $|z|=r >1$ with positive
  orientation, we have
  \[  \int_{\partial \tilde{D}_2} q_1 \ast dq_2 = \lim_{r \nearrow 1} \int_{\pi^{-1}(F(C_r))} q_1 \ast dq_2.  \]
 \end{theorem}
 Note that $\pi^{-1}(F(C_r))$ is one or two closed analytic curves.
 This theorem also verifies that Definition \ref{de:integral_with_corners}
 is sensible.

 Now we prove that the module is well-defined. Recall
 that if $\alpha$ is a quadratic differential admissible for a conformal disk
 $D_1$ say, and $g:E_1 \rightarrow D_1$ is a conformal bijection, then the pull back
 $g^* \alpha$ preserves trajectories.
 Thus $\alpha$ is admissible for $D_1$ if and
 only if $g^* \alpha$ is admissible for $E_1$.

 Now assume that $D_2 \subseteq D_1$ and $E_2 \subseteq E_1$ are simple with respect to
 $D_1$ and $E_1$ respectively.  Assume also that $g(E_2)=D_2$.  Let $\beta$ be the one-form
 on $\tilde{D}_1$ such that $\beta^2 = \pi^* (g^*\alpha)$ and $\delta$ be the one-form on
 $\tilde{E}_1$ such that $\delta^2 = \pi^* \alpha$, whose existence is guaranteed by Theorem \ref{th:primitive_is_constant},
 and let $x$ and $y$ be their primitives respectively.
  Let $\tilde{g}:\tilde{D}_1 \rightarrow \tilde{E}_1$  be the lift of
 $g$ to the double cover.   It is immediately seen that if
 $\beta = b(z)dz$ and $\delta = c(z) dz$ in a local coordinate, then $b(z) = c(\tilde{g}(z))\tilde{g}'(z)$
 (possibly after switching the sign of $\delta$).
 Thus $x = y \circ \tilde{g}$.

 Set $q_1=\mbox{Re}(x)$ and $p_1 = \mbox{Re}(y)$, so that $(q_1,q_2)$ and $(p_1,p_2)$ are each a harmonic
 pair (with a definite choice of sign) induced by $(D_1,D_2,g^*\alpha)$ and $(E_1,E_2,\alpha)$ respectively.
 it is clear that if $p_1-p_2$ is harmonic then $p_1 \circ \tilde{g} - p_2 \circ \tilde{g}$ is harmonic.
 Thus since $q_1 = p_1 \circ \tilde{g}$ we have $(q_1,q_2)=(p_1,p_2)$.

 Next, observe that by Proposition \ref{pr:cover_uniqueness}, given two
 distinct covers $\tilde{D}_1$ and $\hat{D}_1$ of $D_1$, there is a
 conformal map $\phi:\tilde{D}_1 \rightarrow \hat{D}_1$.    The integral
 \ref{eq:def_module} is invariant under $\phi$ by (\ref{eq:star_change_o_coord}),
 so the module $m(D_1,D_2,Q(z)dz^2)$ is well-defined.
 Similarly, applying (\ref{eq:star_change_o_coord}) we see that the integral (\ref{eq:def_module})
 is invariant under composition by $\tilde{g}$.  Thus we have
 proven the following theorem.
 \begin{theorem} \label{th:conformal_invariance}
  Let $D_1$ be a conformal disk, and $D_2$ be simple subdomain.  Let $g:D_1 \rightarrow E_1$ be
  a conformal bijection and $E_2= g(D_2)$.  If $\alpha$ is admissible for $E_1$ then
  \[  m(D_1,D_2,g^*\alpha) = m(E_1,E_2,\alpha).  \]
 \end{theorem}
 In other words, the modules are conformally invariant.

 \begin{remark}[convention for disconnected cover, part two] \label{re:disconnected_doubling}
 In the case that $\alpha$ has no zeros and poles
 of odd order, $\tilde{D}_1$ is disconnected.  In that case there is a meromorphic
 one-form $\beta$ on $D_1$ such that $\beta^2 = \alpha$, and the contour integral defining the module
 reduces to a contour integral over $\partial D_2$
 as follows.  Let $y_1$
 denote the primitive on $D_1$, and $p_1 = \mathrm{Re}(y_1)$; similarly define $p_2$ such
 that $p_1-p_2$ is harmonic and $p_2 =0$ on $\partial D_2$.  It is clear that $(p_1,p_2)$ is
 the restriction to $D_1$ of one of the harmonic pairs $\pm(q_1,q_2)$.  Thus we have
 \[  m(D_1,D_2,\alpha) = 2 \int_{\partial D_2} p_1 \ast d p_2.   \]
 We will see ahead that this special case includes the functionals defined by Nehari.
 \end{remark}
\end{subsection}
\begin{subsection}{Monotonicity theorems}  \label{se:monotonicity}
 In this section we prove the main result of the paper.

 We recall Green's identities, which we will need in their general form on Riemann surfaces (in order to apply
 them on double covers).  Let $u$ and $v$ be functions on a bordered Riemann surface $R$ bounded
 by analytic curves, which are $C^2$ functions on the closure of $R$.  We then have
 \begin{equation} \label{eq:Greens_identity}
   \int_{\partial R} v \ast du = \iint_{R} v \wedge \ast u + \iint_R v \cdot d \ast d u.
 \end{equation}
 which implies another Green's identity
 \begin{equation} \label{eq:Greens_third}
  \int_{\partial R} \left( v \ast du  -u \ast dv \right) = \iint_R \left( v \cdot d \ast d u - u  \cdot d \ast d v   \right).
 \end{equation}
 In local coordinates $z=x+iy$ observe that
 \[  d u \wedge \ast dv = \frac{1}{2i} \left( u_x v_x + u_y v_y  \right) d\bar{z} \wedge dz \]
 so in particular if $u=v$ the integral is the familiar Dirichlet energy of $u$ and must be non-negative.

 If there is a global coordinate $z$ on $R$, denoting infinitesimal arc length by $ds$ and the unit
 outward normal by $n$, these have the form
 \begin{equation*} \label{eq:Greens_identity_local}
   \int_{\partial R} v \frac{\partial u}{\partial n} ds = \iint_{R} \nabla u \cdot \nabla v dA + \iint_R v \triangle u \, dA
 \end{equation*}
 and
 \begin{equation*} \label{eq:Greens_third_local}
  \int_{D} \left( v \frac{\partial u}{\partial n} -u \frac{\partial v }{\partial n} \right)\,ds = \iint_D \left( v \triangle u- u \triangle v   \right)\,dA
 \end{equation*}
 where $dA$ is the measure $d\bar{z} \wedge dz/2i$.
 \begin{lemma} \label{le:Dirichlet_energy_finite}
  Let $D_1$ be a conformal disk and $D_2 \subseteq D_1$ be a simple subdomain.  Let $\alpha$ be a quadratic
  differential admissible for $D_1$, and let $\pm (q_1,q_2)$ be the corresponding harmonic pair on
  Let $D_3$ be any open subset of $D_2$ which contains all of the poles of $\alpha$.  For any double cover
  $\tilde{D}_1$ adapted to $\alpha$,
  \[  \iint_{\tilde{D}_i \backslash \tilde{D}_3} d q_i \wedge \ast dq_i < \infty  \]
  for $i=1,2$.
 \end{lemma}
 \begin{proof}
  Fix $i$.
  Let $\Omega$ be a domain bounded by analytic Jordan curves, whose closure is contained in
  $\tilde{D}_3$ and which contains all the singularities of $q_i$.
  Let $F:\disk \rightarrow D_i$ be a conformal bijection and $C_r$ the curve $|z|=r$
  with positive orientation.  By  Green's identity (\ref{eq:Greens_identity})
  we have
  \begin{align*}
   \iint_{\tilde{D}_i \backslash \tilde{D}_3} d q_i \ast d q_i  &
    \leq \iint_{\tilde{D}_i \backslash \Omega} d q_i \ast d q_i =
   \lim_{r \nearrow 1} \iint_{\pi^{-1} \circ F(\{ |z|<r \})  \backslash \Omega} d q_i \ast d q_i \\
   & = \lim_{r \nearrow 1} \int_{\pi^{-1} \circ F(C_r)} q_i \ast d q_i -
    \int_{\partial \Omega} q_i \ast d q_i
   \\ &  = - \int_{\partial \Omega} q_i \ast d q_i
  \end{align*}
  where the last equality follows from Theorem \ref{th:convergence_and_nice_approximation}.  Since $q_i$ is harmonic
  on an open set containing $\partial \Omega$ the final integral exists.
 \end{proof}

 \begin{theorem}[Positivity] \label{th:positivity} Let $D_1$ be a conformal disk and $D_2 \subseteq D_1$ a simple subdomain,
 and let $\alpha$ be admissible for $D_1$.  Then
  \[ m(D_1,D_2,\alpha) \leq 0.  \]
 \end{theorem}
 \begin{proof}  Assume for now that $D_2$ is bounded by an analytic curve in $D_1$.
  Let $(q_1,q_2)$ be one of the harmonic pairs induced by $(D_1,D_2,\alpha)$.   Finally let
  \[   w(p) = \left\{ \begin{array}{cc} q_1(p) & p \in \tilde{D}_1 \backslash \tilde{D}_2 \\
       q_1(p)-q_2(p) & p \in \tilde{D}_2. \end{array} \right.  \]
  Observe that this is well-defined since $q_1 - q_2$ is single-valued and $q_1$
  is single-valued on $\tilde{D}_1 \backslash \tilde{D}_2$ by Theorem \ref{th:primitive_is_constant}.
  In that case, using Green's identity (\ref{eq:Greens_identity}),
  \begin{eqnarray} \label{eq:positivity}
   \iint_{\tilde{D}_1} dw \wedge \ast dw & = & \iint_{\tilde{D}_1 \backslash \tilde{D}_2} dw \wedge \ast dw
    + \iint_{\tilde{D}_2} dw \wedge \ast dw   \nonumber \\
    & = & \int_{\partial \tilde{D}_1} q_1 \ast dq_1 -
    \int_{\partial \tilde{D}_2} q_1 \ast dq_1
    + \int_{\partial \tilde{D}_2} (q_1 - q_2) \ast d( q_1 - q_2)
    \nonumber \\ & = & - \int_{\partial \tilde{D}_2} q_1 \ast dq_2.
  \end{eqnarray}
  Since the left hand side is greater than or equal to zero this
  completes the proof in the case that the boundary of $D_2$ is analytic.

  For the general case, let $F_i:\disk \rightarrow D_i$ be conformal bijections for $i=1,2$.  In the computation
  above replace
  $\partial \tilde{D}_i$ with $\pi^{-1}F_i(C_r)$, replace
  $\tilde{D}_1 \backslash \tilde{D}_2$ with the region
  bounded by $\pi^{-1} \circ F_1(C_r)$ and $\pi^{-1} \circ F_2(C_r)$, and $\tilde{D}_2$
  by $\pi^{-1} \circ F_2(\{|z|<r \})$.  Letting $r \nearrow 1$ and
  applying Theorem \ref{th:convergence_and_nice_approximation} and
  Lemma \ref{le:Dirichlet_energy_finite} completes the proof.
 \end{proof}

 We also have the following theorem, which says that the modules have a kind of transitivity property.
 \begin{theorem} \label{th:monotonicity_with_remainder}  Let $D$ be a conformal disk
 and let $D_1$ and $D_2$ be simple domains such that
  $D_2 \subseteq D_1 \subseteq D$.  Let $\alpha$ be a quadratic differential which is admissible for
  $D$ and assume that all of the poles of $\alpha$ are contained in $D_2$.
  Let $\pm(q,q_i)$ be the harmonic pairs induced by $(D,D_i,\alpha)$ for $i=1,2$.
  Then
  \[  m(D,D_2,\alpha) - m(D,D_1,\alpha) = \int_{\partial \tilde{D}_2} q_1 \ast dq_2.  \]
  Thus if $\alpha_1$ is the quadratic differential on $D_1$ induced by $\alpha$ then
  \[  m(D,D_2,\alpha) - m(D,D_1,\alpha) = m(D_1,D_2,\alpha_1).  \]
 \end{theorem}
 \begin{proof}  Assuming that $D_2$ is bounded by an analytic curve, we compute
  \begin{eqnarray*}
   I & = & - m(D,D_1,\alpha) -\int_{\partial \tilde{D}_2} q_1 \ast dq_2 + m(D,D_2,\alpha)  \\
    & = & - \int_{\partial \tilde{D}_1} q  \ast d q_1
    - \int_{\partial \tilde{D}_2} q_1 \ast dq_2 + \int_{\partial \tilde{D}_2} q \ast dq_2 \\
    & = & \int_{\partial \tilde{D}_1} \left(q_1 - q \right) \ast dq_1
    - \int_{\partial \tilde{D}_2} \left(q_1-q \right)  \ast dq_2
  \end{eqnarray*}
  where we have used the fact that $q_1=0$ on $\partial \tilde{D}_1$.
  By Green's identity (\ref{eq:Greens_third})
  \begin{eqnarray*}
   \int_{\partial \tilde{D}_1} (q_1 - q)  \ast dq_1 & = &
   \int_{\partial \tilde{D}_1} (q_1 - q)  \ast dq_1
   - \int_{\partial \tilde{D}_1} q_1 \ast d (q_1 -q) \\
   & = & \int_{\partial \tilde{D}_2} (q_1 - q) \ast dq_1
   - \int_{\partial \tilde{D}_2} q_1 \ast d(q_1 - q)
  \end{eqnarray*}
  since $q_1-q$ and $q_1$ are harmonic on $\tilde{D}_1 \backslash \tilde{D}_2$.
  So
  \begin{eqnarray*}
   I & = & \int_{\partial \tilde{D}_2} \left( q_1 - q \right) \ast dq_1
   - \int_{\partial \tilde{D}_2} q_1\ast d(q_1-q)  - \int_{\partial \tilde{D}_2} \left( q_1 - q \right)
   \ast dq_2  \\ & = &
   \int_{\partial \tilde{D}_2} (q_1 - q) \ast d(q_1 - q_2)
    - \int_{\partial \tilde{D}_2} (q_1 - q_2) \ast d(q_1-q) =0
  \end{eqnarray*}
  by Green's identity.
  The general case is handled by approximating with analytic curves as in the proof of Theorem \ref{th:positivity}.

  The final claim follows from Proposition \ref{pr:induced_things_are_modules}.
 \end{proof}
  Theorems \ref{th:positivity} and \ref{th:monotonicity_with_remainder} immediately imply that the
 higher-order reduced modules are monotonic in the following sense.
 \begin{corollary}[Monotonicity] \label{co:monotonicity}
  If $D$, $D_1$ and $D_2$ are simple domains satisfying $D_2 \subseteq D_1 \subseteq D$ and $\alpha$ is a quadratic differential
  admissible for $D$ then
  \[  m(D,D_2,\alpha) \leq m(D,D_1,\alpha).  \]
 \end{corollary}
\end{subsection}
\begin{subsection}{The case of equality}
 Equality in Theorem \ref{th:positivity} occurs only if $\alpha$ is also admissible for
 $D_2$.
 \begin{theorem}  \label{th:equality}Equality occurs in Theorem \ref{th:positivity} if and only if
  $D_1 \backslash D_2$ consists of a network of trajectories of $\alpha$.
 \end{theorem}
 \begin{proof}
  By equation (\ref{eq:positivity}) equality holds if and only if
  \[  \iint_{\tilde{D}_1} |\nabla w|^2\,dA = 0. \]
  So in particular we have that
  \[  \iint_{\tilde{D}_1 \backslash \tilde{D}_2} |\nabla q_1|^2\,dA  =0; \]
  since $q_1$ is not constant this can only occur if $\tilde{D}_1 \backslash \tilde{D}_2$
  has measure zero.  We already know that $\tilde{D}_2$ is bounded by a network of analytic
  arcs, since $D_2$ is simple.  Thus $\tilde{D}_1 \backslash \tilde{D}_2$
  is a network of a finite number of analytic arcs $\gamma_1,\ldots,\gamma_k$.  We must also have that
  \[  \iint_{\tilde{D}_2} |\nabla (q_1 - q_2)|^2\,dA  =0   \]
  which can only occur if $q_1 - q_2$ is constant (and hence zero, since
  $\partial \tilde{D}_1$ and $\partial \tilde{D}_2$ must have points in common).
  Since $q_2=0$ on $\gamma_i$ for $i=1,\ldots,k$ we must also have that $q_1=0$
  on those curves.  By Proposition \ref{pr:trajectories_and_constants}
  the curves $\pi \circ \gamma_i$ are trajectories of $\alpha$.
 \end{proof}
 \begin{remark}  \label{re:functional_extension}
  Assuming only that $D_1$ is a conformal disk
  one can extend the functional $m(D_1,D_2,\alpha)$ to general simply connected
  domains $D_2 \subseteq D_1$ in various ways (for example, by writing it in
  terms of derivatives of Green's function \cite{Schippers_conf_inv_analyse}
  or approximating the simply
  connected domain by simple ones).  This is easily done for arbitrary
  quadratic differentials with finitely many poles and zeros, although we will
  not show this here.

  Furthermore, by conformal invariance we can without loss of generality assume that
  $D_1=\disk$ and $D_2 = f(\disk)$ for a conformal map $f$.
  Applying the Schiffer variational technique  \cite[Theorem II.29]{Oli_thesis}
  for bounded univalent functions
  shows that the extremal function for an extended functional
  maps onto the disk minus trajectories of a quadratic
  differential; that is, the inner domain is simple.  Thus it can be shown that
  Theorem \ref{th:equality} holds for the extended functional.

  We will
  not pursue the general extension and equality case here.    Instead,
  in Section \ref{se:examples}, we will give specific functionals which
  extend by inspection to functionals for arbitrary simply connected $D_2$,
  which are continuous with respect to uniform convergence on compact sets.
  Furthermore they can easily be shown to satisfy the conditions of \cite[Theorem II.29]{Oli_thesis},
  so Theorem \ref{th:equality} holds for these specific functionals.
 \end{remark}

 We also have the following elementary consequence of monotonicity and
 boundedness.
 \begin{corollary}  Let $D_1$ be a conformal disk and $D_2 \subseteq D_1$ a simple domain.
  Let $\alpha$
  be a quadratic differential which is admissible for $D_1$.
  If $m(D_1,D_2,\alpha)=0$, then $m(D_1,D_3,\alpha)=0$ for all domains
  $D_3$ such that $D_2 \subseteq D_3 \subseteq D_1$.
 \end{corollary}
 \begin{proof}
  This follows immediately from Corollary \ref{co:monotonicity} and Corollary \ref{th:positivity},
  once we know that $D_3$ is simple.  However that follows immediately from the fact that
  $D_2$ is simple as a consequence of Theorem \ref{th:equality}.
 \end{proof}
\end{subsection}
\begin{subsection}{Relation to Nehari's general inequality}
 Nehari's monotonicity theorem \cite{Nehari_some_inequalities} says the following, in the simply connected
 case. (The domains in \cite{Nehari_some_inequalities} are assumed to have analytic boundary curves, whereas
 here they might be only piecewise analytic).
 \begin{theorem}[Nehari monotonicity theorem, simply connected case] \label{th:Nehari_monotonicity_theorem}
  Let $D$ be a conformal disk and $D_1$ and $D_2$ be
  simple domains such that $D_2 \subseteq D_1 \subseteq D$.  Let $S$ be a harmonic
  function on $D$, except possibly for finitely many singularities in $D_2$.    Let $p_i$ be
  the unique functions on $D_i$ such that $p_i=0$ on $\partial D_i$ and $S+p_i$
  is harmonic on $D_i$.  Then
  \[  \int_{\partial D_2}  S \frac{\partial p_2}{\partial n} ds \geq \int_{\partial D_1} S
   \frac{\partial p_1}{\partial n} ds.  \]
 \end{theorem}
 Thus defining the functional
 \[  M(D,D_1,S)=\int_{\partial D_1} S
   \frac{\partial p_1}{\partial n} ds  \]
 we have that $M(D,D_2,S) \geq M(D,D_1,S)$
 whenever $D_2 \subseteq D_1 \subseteq D$.  We call this functional the
 ``Nehari functional''.  If we let $p$ be the unique
 function on $D$ such that $S+p$ is harmonic and $p=0$ on $\partial D$, then
 the lower bound of this functional is
 \[  M(D,D,S) = \int_{\partial D} S \frac{\partial p}{\partial n} ds.  \]
 It can be shown using Green's identity \cite{Schippers_derivative_Nehari}
 that
 \[   M(D,D_1,S) - M(D,D,S)  =\int_{\partial D_1} S \frac{\partial p_1}{\partial n} \,ds - \int_{\partial D} S \frac{\partial p}{\partial n} \,ds
    = -\int_{\partial D_1} p \frac{\partial p_1}{\partial n} \,ds.  \]
 Since $p - p_1 =(p+S)-(S+p_1)$ is harmonic on $D_1$ we can rewrite this as
 \begin{equation} \label{eq:constant_term}
   M(D,D_1,S) - M(D,D,S) = M(D,D_1,-p).
 \end{equation}
 Note that $M(D,D,-p)=0$.

 This leads to the following Proposition.
 \begin{proposition} \label{pr:Nehari_special_case}  Let
 $D$ be a conformal disk and $D_1 \subseteq D$ a simple subdomain, and
 $S$ be a singularity function on $D$ as in Nehari's theorem \ref{th:Nehari_monotonicity_theorem}.
  Let $M(D,D_1,S)$ be the resulting Nehari functional.  Letting $p$ be the unique
  function on $D$ such that $p=0$ on $\partial D$ and $S+p$ is harmonic, setting $\hat{S}=-p$ we have
  $M(D,D_1,S)= M(D,D_1,\hat{S}) + C$ where $C$ is independent of $D_1$.
  With this choice of singularity function $M(D,D,\hat{S})=0$.
  Furthermore
  \[  \alpha = 4 \left( \partial p \right)^2  \]
  is admissible for $D$ and
  \[  M(D,D_1,\hat{S}) = - \frac{1}{2} m(D,D_1,\alpha).  \]
 \end{proposition}
 \begin{proof}  To be consistent with Nehari's notation, without loss of generality
  we choose $D = \disk$ and set
  \[  \alpha = Q(z)dz^2 = 4 \frac{\partial p}{\partial z} dz^2.  \]
  The first two claims were proven above (equation (\ref{eq:constant_term})
  and immediately following).

   To see that $Q(z)dz^2$ is admissible for $D$, observe that
  since $p$ is zero on $\partial D$
  \[  \frac{\partial p}{\partial s} \,ds =0  \]
  along $\partial D$, and thus $\partial p/\partial z$ is pure imaginary on $\partial D$.
  Thus $Q(z)dz^2 \leq 0$ on $\partial D$, that is $\partial D$ is a trajectory of $Q(z)dz^2$.

  The final claim follows from Remark \ref{re:disconnected_doubling} once we observe that $p-p_1 = p+S - (p_1+S)$ is harmonic
  on $D_1$ and $p_1=0$ on $\partial D_1$; hence $\pm(p,p_1)$ is the harmonic pair associated with $(D,D_1,Q(z)dz^2)$.
 \end{proof}

 In other words, without loss of generality, we can assume that the singularity function in
 Nehari's theorem is the primitive of the square root of a quadratic differential.  This
 shows that Corollary \ref{co:monotonicity} significantly generalizes the simply connected case of
 Nehari's theorem \ref{th:Nehari_monotonicity_theorem}
 by removing the requirement that the quadratic differential be a perfect square.
 It is evident that the techniques of this paper can also be used to  extend Nehari's theorem
 in the case of finitely connected domains.
 \begin{remark}
  In Schiffer's method, the quadratic differential is generated by taking a functional
  derivative of a fixed functional.  The differential is not completely determined by the
  functional and depends on the extremal function; furthermore admissibility of the
  map is necessary but not sufficient for extremality.  In this paper, the quadratic differential is
  fixed, and admissibility is necessary and sufficient.
 \end{remark}
 \begin{remark}
  It is natural to ask whether the functionally derivative of
  $m(D,D_1,\alpha)$ can be written in terms of $\alpha$. This is
   indeed true in the case that $\alpha$
  is a perfect square  with pole at the
  origin \cite{Schippers_derivative_Nehari}.
 \end{remark}
\end{subsection}
\end{section}
\begin{section}{Growth theorems}  \label{se:examples}
In this section we use Theorem \ref{th:positivity} and Corollary \ref{co:monotonicity} to derive a family
of growth theorems for bounded univalent functions.
\begin{subsection}{Preliminary computations}
  In this section we collect some computations that will be useful in the proof of the main application.

  First, we will write the functional in a form which is easier to compute.  In the following, we
  express the contour integrals in terms of a local parameter $z$.  For ease of presentation,
  we assume
  that there is a global parameter $z$; if the parameter is only local, the expressions are still valid
  along some subcontour.
 For any real function $h$ we introduce the notation
 \[  \frac{\partial h}{\partial s} ds = \frac{\partial h}{\partial x} \, dx + \frac{\partial h}{\partial y}\,dy. \]
  We may thus write
 \begin{equation} \label{eq:dbydz_ds_and_dn_identity}
  2 \frac{\partial h}{\partial z} \,dz = \frac{\partial h}{\partial s} ds + i \frac{\partial h}{\partial n} ds
 \end{equation}
 where $ds$ is infinitesimal arc length and $n$ is the normal to the right of direction of motion.

 We will find a more computable expression for the conformal invariants.  Let $D$ be a
 conformal disk and $D_1$ be a simple domain
 such that $D_1 \subseteq D$; let $\alpha$ be a quadratic differential canonically
 admissible for $D$.  Let $\tilde{D}$ denote the double cover of $D$ adapted to $Q(z)dz^2$.  Let $\pm(q,q_1)$
 denote the harmonic pairs induced by $(D,D_1,\alpha)$, and $x$, $x_1$ be the analytic functions
 whose single-valued positive real parts are $q$, $q_1$ respectively.

 Since $q_1 = \mbox{Re}(x_1)=0$ on $\partial \tilde{D}_1$,
 \begin{eqnarray*}
   m(D,D_1,\alpha) & = & \int_{\partial \tilde{D}_1} q \frac{\partial q_1}{\partial n} \,ds
    = \mbox{Re}\left( \int_{\tilde{D}_1} x
   \frac{\partial q_1}{\partial n} \,ds \right) \\
   & = & \mbox{Re} \left( \int_{\partial \tilde{D}_1} \left( x - x_1 \right)  \frac{\partial q_1}{\partial n} \,ds \right).
 \end{eqnarray*}
 Again using the fact that $q_1=0$ on $\tilde{D}_1$, we have that
 \[  \frac{\partial q_1}{\partial s} \,ds =0  \]
 along $\partial \tilde{D}_1$ so
 \[  \frac{\partial q_1}{\partial n}\,ds = \frac{2}{i} \frac{\partial q_1}{\partial z} \,dz  \]
 along $\partial D_1$ by (\ref{eq:dbydz_ds_and_dn_identity}).   Furthermore by the Cauchy-Riemann
 equations
 \[  \frac{\partial x_1}{\partial z} = 2 \frac{\partial q_1}{\partial z}  \]
 so
 \begin{equation*}
  m = \mbox{Re} \left( \frac{1}{i} \int_{\partial \tilde{D}_1} \left( x - x_1 \right)  \frac{\partial x_1}{\partial z} \,dz \right).
 \end{equation*}
 Furthermore, since $x - x_1$ and its $z$-derivative are non-singular analytic functions,
 \[ \int_{\partial \tilde{D}_1} \left( x - x_1 \right) \left( \frac{\partial x_1}{\partial z} - \frac{\partial x}{\partial z}
 \right)\,dz =0. \]
 Thus we have the identity
 \begin{equation} \label{eq:q_integral_identity}
   m(D,D_1,Q(z)dz^2) = \int_{\partial \tilde{D}_1} q \frac{\partial q_1}{\partial n} \,ds =  \mbox{Re} \left( \frac{1}{i}
   \int_{\partial \tilde{D}_1} \left( x - x_1 \right)  \frac{\partial x}{\partial z} \,dz \right).
 \end{equation}
\end{subsection}
\begin{subsection}{A general two-point growth theorem} \label{se:growth_theorems}
 We consider the following quadratic differential,
 for $r>0$:
 \begin{equation} \label{eq:Q_for_growth_all}
   Q(z) dz^2  = -  e^{- i \psi} \frac{(e^{i \psi} + z)^2}{z^2 (z-r)(z-1/r)} dz^2.
 \end{equation}
 For $z=e^{i \theta}$ we have that
 \[  Q(z) = r e^{-2i \theta} \frac{ 2 \cos{\left( \psi/2 - \theta /2 \right)} + 2}
   {|1- r e^{i \theta}|^2}   \]
 so $Q(z) dz^2$ is admissible for $\disk$.  If we can compute this module, then
 by conformal invariance we know the module of any quadratic differential
 with a double pole and simple zero in the interior, and double zero on $\partial \disk$.

  We require a double cover of $\disk$ branched at $r>0$ on which to compute $x$; we choose
 $\tilde{\disk} = \disk$ and
 \begin{align*}
  \pi:\tilde{\disk} & \rightarrow \disk \\
   \zeta & \mapsto T(\zeta^2)
 \end{align*}
 where
 \[  T(w)= \frac{w+r}{1+rw}.  \]
 Next we pull back the quadratic differential to the cover:
 \begin{align} \label{eq:Q_growth_all_on_cover}
  \pi^*( Q(z) dz^2) & = - \frac{e^{-i \psi} \left( e^{i \psi} + (\zeta^2 + r)/(1+r \zeta^2) \right)^2}
    {  \left( \frac{\zeta^2 + r}{1 + r \zeta^2} \right)^2 \left( \frac{\zeta^2 +r}{1+r \zeta^2}
     -r \right) \left( \frac{\zeta^2 + r}{1 + r \zeta^2} - \frac{1}{r} \right)}
     \cdot \frac{4 \zeta^2 (1-r^2)^2}{(1+r \zeta^2)^4} d\zeta^2 \nonumber \\
     & = 4 r (1 + r e^{i \psi})^2 e^{-i \psi} \frac{ (T(e^{i \psi}) + \zeta^2)^2}
     { (1 + r \zeta^2)^2 (\zeta^2 + r)^2} d\zeta^2.
 \end{align}
 \begin{theorem}  \label{th:growth_all}
  Let $Q(z)dz^2$ be given by (\ref{eq:Q_for_growth_all}).
  Then for a one-to-one conformal mapping  $f:\disk \rightarrow \disk$ such
  that $f(0)=0$ and $r \in f(\disk)$ we have that
  \begin{equation}
   m(\disk,f(\disk),Q(z)dz^2)  =  4 \pi \text{Re} \left[ - e^{i \psi}
     \log{\left( \frac{w_0 f'(0)}{f(w_0)} \frac{ 1 - |f(w_0)|^2}{1-|w_0|^2}
     \right)} \right] + \log{\left( \frac{ 1+|f(w_0)|}{1- |f(w_0)|} \frac{1 -|w_0|}{1+|w_0|}
     \right)}
  \end{equation}
  where $w_0 = f^{-1}(r)$, the first term uses the unique branch of logarithm such that
  $\log{(w_0 f'(0)/f(w_0))} \rightarrow 0$ as $w_0 \rightarrow 0$ and the second term
  uses the principal
  branch.
 \end{theorem}
 \begin{proof}
  By (\ref{eq:Q_growth_all_on_cover}) we have that
  \begin{align} \label{eq:xprime_growth_all}
    x'(\zeta) & = 2 \sqrt{r} e^{-i\psi/2} \frac{(1+ r e^{i \psi})(T(e^{i\psi}) + \zeta^2)}
     {(1+r\zeta^2)(\zeta^2 + r)}  \nonumber \\ \nonumber
    & = \frac{2 \sqrt{r} e^{i \psi/2}}{\zeta^2 + r} + \frac{2 \sqrt{r} e^{-i\psi/2}}{1+ r \zeta^2} \\
    & = \frac{ i e^{i \psi/2}}{ \zeta + i \sqrt{r}} - \frac{i e^{i \psi/2}}{\zeta - i \sqrt{r}}
     + \frac{\sqrt{r} e^{- i \psi/2}}{1 + i \sqrt{r} \zeta} + \frac{\sqrt{r} e^{-i \psi/2}}
     {1 - i \sqrt{r} \zeta}.
  \end{align}
  Thus setting $\beta = i e^{i \psi/2}$ we have
  \begin{equation} \label{eq:x_for_general_growth}
    x(\zeta) = \beta \log{\frac{ \zeta + i \sqrt{r}}{\zeta - i \sqrt{r}}} - \overline{\beta}
     \log{ \frac{ 1 - i \sqrt{r} \zeta}{1 + i \sqrt{r} \zeta}}
  \end{equation}
  where we choose the principal branch of the logarithm.
  This requires some justification.
  Observing that $\zeta \mapsto (\zeta + i \sqrt{r})/(\zeta - i \sqrt{r})$
  maps $\{ |\zeta| > \sqrt{r} \}$ onto the right half plane, and
  $\zeta \mapsto (1- i \sqrt{r} \zeta)/(1+ i \sqrt{r} \zeta)$ maps $\{ |\zeta| < 1/\sqrt{r} \}$ onto the right half plane, so this function is
  well-defined.  It is easily checked that it is a primitive of $x'$.
  Since $x(1)=0$, it follows from the fact that $\partial \disk$ is a
  trajectory of $Q(z) dz^2$ that $\text{Re}(x)=0$ on $\partial \tilde{\disk}$.  It
   is also easily shown directly using $\bar{\zeta}=1/\zeta$.

   Now let $D_1$ be a simple subdomain of $\disk$ containing $0$ and $r$,
   and $\tilde{D_1}$ be $\pi^{-1}(D_1)$.
   We will find an explicit formula for $x_1$ in terms of conformal maps.
   We use the following notation.  Let $f:\disk \rightarrow D_1$ be a conformal bijection
   such that $f(0)=0$.  Let $G = S^{-1} \circ f^{-1} \circ T$ for
   \[  S(w) = \frac{w + w_0}{ 1 + \overline{w_0}w}.  \]
   It can then be checked that $G(-r)=-w_0$ and $G(0)=0$.  Furthermore $G$ is a conformal
   bijection from $T^{-1}(D_1)$ onto $\disk$.  Finally set $\tilde{G} = \sqrt{G(\zeta^2)}$,
   so that $\tilde{G}$ is a conformal map from $\tilde{D}_1$ onto $\tilde{\disk}$ such that
   $\tilde{G}(0)=0$.

   It follows immediately from the analysis of $x$ and the properties of $\tilde{G}$
   that
   \[  x_1(\zeta) = \beta \log{ \left[ \frac{\tilde{G}(\zeta) - \tilde{G}(-i \sqrt{r})}
    {\tilde{G}(\zeta) - \tilde{G}(i \sqrt{r})} \right]} -\overline{\beta} \log{ \left[
    \frac{1 - \overline{\tilde{G}(-i \sqrt{r})} \tilde{G}(\zeta)}
    {1 - \overline{\tilde{G}(i \sqrt{r})} \tilde{G}(\zeta)} \right]}  \]
    where again we use the principal branch of logarithm.

    Note that for the
    principal branch of logarithm, whenever $z$ and $w$ are both in the right half plane,
    it holds that $\log{z w} = \log{z} + \log{w}$ with no correction of the branch.  So we can
    write
    \begin{align} \label{eq:temp_log_branch_o}
      x(\zeta) - x_1(\zeta) & =   \beta \log{  \left[ \frac{\zeta+i \sqrt{r}}{\zeta - i \sqrt{r}}
      \cdot \frac{\tilde{G}(\zeta) - \tilde{G}(i \sqrt{r})}{\tilde{G}(\zeta) - \tilde{G}(-i \sqrt{r})}
    \right]} 
    - \overline{\beta} \log \left[ \frac{1 - i \sqrt{r} \zeta}{1 + i \sqrt{r} \zeta}
    \cdot  \frac{1 - \overline{\tilde{G}(i \sqrt{r})} \tilde{G}(\zeta)}
    {1 - \overline{\tilde{G}(- i \sqrt{r})} \tilde{G}(\zeta)} \right].
    \end{align}
    Observe that
    \[   \lim_{r \rightarrow 0}  \left[ \frac{\zeta+i \sqrt{r}}{\zeta - i \sqrt{r}}
      \cdot \frac{\tilde{G}(\zeta) - \tilde{G}(i \sqrt{r})}{\tilde{G}(\zeta) - \tilde{G}(-i \sqrt{r})}
    \right] = 1  \]
    so since we are using the principal branch of log the first term must approach $0$ as
    $r \rightarrow 1$, and similarly for the second term.

   Using (\ref{eq:q_integral_identity}) and (\ref{eq:xprime_growth_all}) we see that
   \begin{align} \label{eq:temp_log_branch}
    m & = 2 \pi \text{Re} \left( \beta \left. ( x- x_1) \right|_{-i \sqrt{r}}
      - \beta \left. ( x- x_1) \right|_{i \sqrt{r}} \right) \nonumber \\
      & = 4 \pi \text{Re} \left( - \beta^2 \log{ \frac{ \tilde{G}'(i\sqrt{r}) i \sqrt{r} }
      {\tilde{G}(i \sqrt{r})}} \right) +
      \log{\left[ \frac{1 + r}{1-r} \cdot \frac{ 1- |\tilde{G}( i \sqrt{r})|^2}
        {1+ |\tilde{G}( i \sqrt{r})|^2} \right]}
   \end{align}
   where we have repeatedly used $\tilde{G}(-i\sqrt{r}) = - \tilde{G}(i\sqrt{r})$
   and $\tilde{G}'( - i\sqrt{r} ) = \tilde{G}'( i \sqrt{r})$.  Note that the first
   term goes to $0$ as $r \rightarrow 1$ since this holds for
   (\ref{eq:temp_log_branch_o}), and the second term uses the
   principal branch of logarithm.   Using
   \[  \tilde{G}(i \sqrt{r}) = \sqrt{G(-r)} \text{  and  } \tilde{G}'(i\sqrt{r}) =
   \frac{i\sqrt{r} G'(-r)}{\sqrt{G(-r)}}   \]
   we obtain
   \[   m = 4 \pi \text{Re}\left[ - \beta^2 \log{\frac{(-r) G'(-r)}{G(-r)}} \right] + \log{ \left[ \frac{1+r}{1-r}
     \frac{1-|G(-r)|}{1+|G(-r)|} \right]} \]
   and by the definition of $G$, $w_0 = - G(-r)$ and $|f(w_0)|=f(w_0)=r$ we obtain
   \[  G'(-r) = \frac{1-|w_0|^2}{(1-|f(w_0)|^2)f'(0)}  \]
   and so
   \[  m = 4 \pi \text{Re} \left[ \beta^2 \log{\left( \frac{w_0 f'(0)}{f(w_0)}
     \frac{1-|f(w_0)|^2}{1-|w_0|^2} \right)} \right] + \log{ \left( \frac{1 + |f(w_0)|}{1-|f(w_0)|}
      \frac{ 1 - |w_0|}{1+|w_0|} \right)}
      \]
   where the branches of logarithm are as claimed.
 \end{proof}
 \begin{remark}  By conformal invariance, we can obtain an expression for $m(\disk,D_1,Q(z)dz^2)$
  for {\it any} quadratic differential with one simple pole, one double pole, and one double
  zero on the boundary
  by composing with disk automorphisms.
 \end{remark}
 \begin{corollary} \label{co:full_growth}
  Let $f: \disk \rightarrow \disk$ be a one-to-one conformal map such that $f(0)=0$.  Let
  \[  I(f,w) = 4 \pi \text{Re} \left[ - e^{i \psi}
     \log{\left( \frac{w f'(0)}{f(w)} \frac{ 1 - |f(w)|^2}{1-|w|^2}
     \right)} \right] + \log{\left( \frac{ 1+|f(w)|}{1- |f(w)|} \frac{1 -|w|}{1+|w|}
     \right)}  \]
     where the branches are determined as in Theorem \ref{eq:Q_for_growth_all}.  Then $I(f,w) \leq 0$ and
   equality holds at a point $w \in \disk$ if and only if for some $\theta$ $f$ maps
   onto $\disk$ minus trajectories of $e^{-2i\theta}Q(e^{i\theta}z) dz^2$
   where $Q(z)dz^2$ is given by \ref{eq:Q_for_growth_all} for $r = |f(w_0)|$.

    Furthermore, $I(f,w)$ is monotonic in the sense that
   if $f_i:\disk \rightarrow \disk$ are one-to-one conformal maps for $i=1,2$
   such that $f_1(0)=f_2(0)=0$, $f_1(w) =f_2(w)$ and $f_1(\disk) \subseteq f_2(\disk)$
   then $I(f_1,w)\leq I(f_2,w)$.
 \end{corollary}
 \begin{proof}  Assume first that $f(\disk)$ is simple.   Define $\theta$ by $f(w)=r e^{i\theta}$.
  The upper bound of $0$ follows from Theorems \ref{th:growth_all} and \ref{th:positivity},
  applied to the function $f_\theta(z)=e^{-i\theta}f(z)$, using the fact that $I(f,w)=I(f_\theta,w)$.
  Equality holds if and only if $f_\theta$ maps onto $\disk$ minus trajectories of $Q(z)dz^2$;
   that is, if and only if $f$ maps onto the disk minus trajectories of $e^{-2i\theta} Q(e^{i\theta}z)$.  Monotonicity
   similarly follows from Corollary \ref{co:monotonicity}.

   The general claim follows from Proposition \ref{th:density} and Remark \ref{re:functional_extension}.
 \end{proof}
 Note that we could have also phrased the proof in terms of the quadratic differential $e^{-2i\theta} Q(e^{i\theta} z) dz^2$
 rather than applying a rotation to the function.

 We now give some special cases of this theorem.
\end{subsection}
\begin{subsection}{Pick growth theorems and two-point distortion theorem of Ma and Minda}
  Theorem \ref{th:growth_all} implies the classical growth estimates for bounded
  univalent functions.  Choosing $e^{i \psi} =- 1$ we obtain
  \begin{equation} \label{eq:upper_growth_normalized}
    I_{lower}(f,w) = \log{ \left[ \frac{|w||f'(0)|}{(1+|w|)^2} \frac{(1 + |f(w)|)^2}{|f(w)|} \right]}
     \leq 0
  \end{equation}
  which after exponentiating becomes
  \[  |f'(0)| \cdot \frac{|w|}{(1+|w|)^2} \leq \frac{|f(w)|}{(1 + |f(w)|)^2}  \]
  which is equivalent to the lower bound in Pick's growth theorem for bounded univalent
  functions.  Similarly, choosing $e^{i\psi}=1$ we obtain
  \begin{equation} \label{eq:lower_growth_normalized}
   I_{upper}(f,w) = \log{ \left[ \frac{|f(w)|}{(1-|f(w)|)^2} \cdot \frac{(1-|w|)^2}{|w| |f'(0)|}
    \right]} \leq 0
  \end{equation}
  whose exponential gives
  \[   \frac{|f(w)|}{(1-|f(w)|)^2}  \leq   |f'(0)| \cdot \frac{|w|}{(1-|w|)^2}  \]
  which is equivalent to the upper bound in Pick's growth theorem.  However the result is
  stronger in
  that the quantities (\ref{eq:upper_growth_normalized}) and (\ref{eq:lower_growth_normalized})
  are in fact monotonic in the sense of Theorem \ref{th:growth_all}.  That is, given
 one-to-one conformal maps $f_i:\disk \rightarrow \disk$ such that $f_1(0)=f_2(0)=0$,
 $f_1(w)=f_2(w)$, and $f_1(\disk) \subseteq f_2(\disk)$ then $I_{lower}(f_1,w) \leq I_{lower}(f_2,w)$,
 and similarly for $I_{upper}$.

  In fact, we can write the growth theorems in a form due to Ma and Minda (somewhat
  modified, algebraically).  In this case the module $m$  can be written in terms
  of more familiar conformal invariants.  Let
  \begin{equation}  \label{eq:hyperbolic_disk}
   \lambda_{\disk}(u) = \frac{1}{1-|z|^2} ;  \ \ \ \ d_{\Omega}(u,v) = \frac{1}{2} \log \frac{ 1+ \left|(u-v)/(1-\bar{v}u)\right|}
     {1- \left|(u-v)/(1-\bar{v}u)\right|}
  \end{equation}
  denote the hyperbolic line element and distance functions on $\disk$ respectively.  Let $\lambda_{D_1}$
  and $d_{D_1}$ denote the hyperbolic line element and distance function on $D_1$, which can be explicitly
  written for a conformal bijection $F:D_1 \rightarrow \disk$ as
  \begin{equation}  \label{eq:hyperbolic_general}
   \lambda_{D_1}(v) = \frac{|F'(v)|}{1-|F(v)|^2}   \ \ \ \text{and} \ \ \ d_{D_1}(u,v) = d_{\disk}(F(u),F(v)).
  \end{equation}
  For simply connected Riemann surfaces $D$ and $D_1$ with hyperbolic metrics, such that $D_1 \subseteq D$, define
  \begin{equation} \label{eq:J_lower_and_upper}
    J_{lower}(D,D_1,u,v) = \frac{e^{-4d_{D_1}(u,v)}-1}{e^{-4d_D(u,v)}-1}
     \cdot \frac{\lambda_{D}(u)}{\lambda_{D_1}(u)} \ \ \ \ \ J_{upper}(D,D_1,u,v) =  \frac{e^{4d_D(u,v)}-1}{e^{4d_{D_1}(u,v)}-1} \cdot
    \frac{\lambda_{D_1}(u)}{\lambda_{D}(u)}.
  \end{equation}
  \begin{theorem}  Let $D$ be a conformal disk and $D_1 \subseteq D$ simple.  Fix points $u$ and $v$ in $D$ and let
   $\rho, \tau \in \partial D$
   be the terminal points of the geodesic through $u$ and $v$, arranged in the order $\rho, u, v,
   \tau$.  Let $\alpha_\rho$ be the unique quadratic differential admissible for $D$
   with a double pole at $u$, a simple pole at $v$, and a double zero at $\rho$ and
   no other zeros or poles in the closure of $D$.  Let $\alpha_\tau$
   be the unique quadratic differential with a double pole at $u$, a simple pole
   at $v$ and a double zero at $\tau$ and no other zeros or poles.  Then
   for any simple domain $D_1$ containing $u$ and $v$,
   \[   {m(D,D_1, Q_\tau(z)dz^2)} = \log J_{lower}(D,D_1,u,v)  \]
   and
   \[  {m(D,D_1,Q_\rho(z)dz^2)} =  \log J_{upper}(D,D_1,u,v).  \]
  \end{theorem}
  \begin{proof}
   Observe that the expressions $J_{lower}(D,D_1,u,v)$ and $J_{upper}(u,v)$ are conformally
   invariant in the sense that if $g:D \rightarrow E$ is a conformal bijection then
   $J_{lower}(g(D),g(D_1),g(u),g(v)$.   Now observe that if $Q(z)dz^2$
   is given by (\ref{eq:Q_for_growth_all}) with the specific value $e^{i\psi}=-1$, then $\alpha_\tau = g^*Q(z)dz^2$ w
   where $g:D \rightarrow \disk$ is a conformal bijection such that $g(\tau)=-1$, $g(v)=0$ and $g(u)=r$.
   By Theorem \ref{th:conformal_invariance} it thus suffices to prove the claim for $D=\disk$, $u=0$,
   $v=f(w)=r$
   (in which case we will have $g(\tau)=-1$ since $g$ is a hyperbolic isometry).  Similarly for $I_{upper}$.

   Let $f:\disk \rightarrow D_1$ be a conformal bijection such that $f(0)=0$.
   By (\ref{eq:hyperbolic_disk}) and (\ref{eq:hyperbolic_general}) with $F=f^{-1}$ we have that
   \[  \lambda_{\disk}(0) = 1, \ \ d_{\disk}(0,f(w)) = \frac{1}{2} \log \frac{ 1+|f(w)|}{1-|f(w)|}   \]
   and
   \[ \lambda_{D_1}(0) = \frac{1}{|f'(0)|} \ \ d_{D_1}(0,f(w)) = \frac{1}{2} \log \frac{ 1+|w|}{1-|w|}.  \]
   So
   \[  J_{lower}(\disk,D,u,v)=\frac{|w||f'(0)|}{(1+|w|)^2} \cdot \frac{(1 +|f(w)|)^2}{|f(w)|}  \]
   and
   \[  J_{upper}(\disk,D,u,v)= \frac{(1-|w|)^2}{|w||f'(0)|} \cdot \frac{|f(w)|}{(1-|f(w)|)^2}.  \]
   This completes the proof.
  \end{proof}

  Since simple domains are dense by Proposition \ref{th:density},
  we have a monotonic, conformally invariant version of the growth theorem of Pick/Ma-Minda.
  \begin{corollary}  Let $D$ and $D_1$ be simply connected hyperbolic Riemann surfaces
   with hyperbolic metrics such that $D_1 \subseteq D$ and let $u,v \in D_1$.
   Then $J_{lower}(D,D_1,u,v) \leq 0$ and $J_{upper}(D,D_1,u,v) \leq 0$ with equality
   if and only if $D_1$ is $D$ minus trajectories of $\alpha_\tau$ or $\alpha_\rho$ respectively.
   Furthermore, if $D_1 \subseteq D_2 \subseteq \disk$ then $J_{upper}(D_1,u,v) \leq J_{upper}(D_2,u,v)$ and
   $J_{lower}(D_1,u,v) \leq J_{lower}(D_2,u,v)$.
  \end{corollary}
 \begin{remark}  It is interesting to observe that the general growth Theorem \ref{th:growth_all}
  interpolates the upper and lower bound in the Pick/Ma-Minda growth theorem, by allowing the zero
  of the quadratic differential to move between the two ends
  of the hyperbolic geodesic passing through the points $0$ and $w$.
 \end{remark}

\end{subsection}
\begin{subsection}{Argument estimates}
 Choosing $\beta^2 = i$ and $\beta^2 = -i$ in Theorem \ref{th:growth_all}, we obtain the
 following monotonic functionals.

 \begin{corollary}  Let $f:\disk \rightarrow \disk$ be one-to-one and satisfy $f(0)=0$.
  Then
  \[ - \text{arg} \left ( \frac{f(w)}{w f'(0)} \right)
     + \log{ \left( \frac{1 + |f(w)|}{1-|f(w)|} \cdot \frac{1- |w|}{1 + |w|} \right)} \leq 0 \]
  and
  \[  \text{arg} \left( \frac{f(w)}{w f'(0)} \right)
   + \log{ \left(  \frac{1-|f(w)|}{1+|f(w)|} \frac{1 +|w|}{1-|w|} \right)}  \leq 0 \]
  where we use the branch of argument such that $\text{arg} [ f(w)/(w f'(0))]$ goes to
   $0$ as $w \rightarrow 0$ and the principal branch of logarithm.  Both expressions are monotonic in $f$
   in the sense of Corollary \ref{co:full_growth}.
 \end{corollary}
 The equality statement can be deduced from Corollary \ref{co:full_growth}.
\end{subsection}
\end{section}
\begin{section}{Appendix: convergence of the contour integral}   \label{se:appendix}
 We need to show that the integral (\ref{eq:def_module}) converges, by proving Theorem
 \ref{th:convergence_and_nice_approximation}.  We do this, and also show that the definition
  is natural, with the help
 of the following parametrization of $\partial \disk$.  Let $F:\disk \rightarrow D_2$ be a conformal map.
  By Theorem \ref{th:boundary_arcs_vertices} $\partial D \backslash \mathfrak{V}$ consists
 of finitely many analytic arcs $B_i$, $i=1,\ldots,k$, with endpoints lying in $\mathfrak{V}$.
 Each analytic arc is a free boundary arc in the sense of Caratheodory \cite[Section 348]{Caratheodory_book_vol_II}.
 Let $F:\mathbb{D} \rightarrow D$ be a conformal bijection onto $D$.  We will need the following lemma.
 Although geometrically it is almost obvious, a careful proof involves many details.

  \begin{lemma} \label{le:distinguished_parametrization} Let $D_1$ be a conformal
  disk and $D_2$ be a simple domain in $D_1$.  Let $\mathfrak{V}$ be the set of
  vertices of $D_2$, and let $\{B_i\}$ be the connected components of $\partial D_2 \backslash \mathfrak{V}$.
  Let $F:\disk \rightarrow D_2$ be a conformal bijection of $\disk$ onto $D_2$.  $F$ has a continuous
  extension $\hat{F}$ to $\overline{\mathbb{D}}$.
   Let
   $e^{i \theta_1}, \ldots, e^{i\theta_m}$ be the elements of $\hat{F}^{-1}(\mathfrak{V})$,
  arranged so that $\theta_1 < \theta_2 < \ldots < \theta_m$.
  Let $A_j= \{ e^{i\theta} \,:\, \theta_j < \theta < \theta_{j+1} \}$ (where we set $\theta_{m+1} = \theta_1 + 2\pi$).
  \begin{enumerate}
   \item For each analytic boundary arc $B_i$ of $\partial D_2$, $\hat{F}$ maps precisely one or two of the arcs $A_1,\ldots,A_m$
   onto $B_i$; if it maps two separate arcs onto $B_i$ it does so with opposite orientation.
   \item The restriction of $\hat{F}$ to each $A_i$ is a one-to-one analytic parametrization.
   \item For each $A_i$, $\pi^{-1} \circ \hat{F}(A_i)$ consists of
   two disjoint analytic arcs of $\partial \tilde{D}_2$, and each branch of $\pi^{-1} \circ \hat{F}^{-1}$
   is an analytic parametrization.
   \item  The collection $\gamma_i$ of arcs $\pi^{-1} \circ \hat{F}(A_j)$
    is a complete set of maximal boundary arcs of $\partial \tilde{D}_2$, with the orientation
    induced by $\pi^{-1} \circ \hat{F}$.
  \end{enumerate}
 \end{lemma}
 \begin{proof}
  By Theorem \ref{th:zero_classification}, for any $p \in \partial \disk$, we can choose a disk
  $B$ centred on $p$ such that the image of $B \cap \disk$ is bounded by a Jordan
  curve consisting of the analytic arc $F(\partial B \cap \disk)$
  and two analytic trajectories of $\alpha$.
  By Carath\'eodory's theorem \cite{Pommerenkeboundary}, $F$ extends continuously to the boundary
  of $B \cap \disk$, and in particular to an open interval
  arc of $\partial \disk$ containing $p$.  This proves that $F$ has a continuous extension
  $\hat{F}$ to $\partial \disk$.

  To prove (2), fix $A_i$ and set $U = \{ r e^{i \theta} \,: \, a_i < \theta < a_{i+1} \ \ \text{and}
   \ \ s< r <1 \}$ for some $0<s <1$ and $\theta_i < a_i < a_{i+1} < \theta_{i+1}$.
   Since $\hat{F}(A_i)$ contains no vertices of $\partial D_2$,
   $F(U)$ is bounded by a Jordan curve consisting of four analytic curves.  In particular, Carath\'eodory's
   theorem implies that the restriction of $F$ to $U$ extends homeomorphically to the boundary of $U$, and thus is
   one-to-one on $A_i$.  Since $\hat{F}(A_i)$ is an analytic
   arc, by the Schwarz reflection principle $\left. F \right|_U$
   extends to a biholomorphism of an open neighbourhood of $A_i$.

 We now prove (1).   Since $\hat{F}(A_i)$ joins two vertices of $\partial D_2$, it must be a surjection onto some $B_j$.
 We show that there is at most two pre-images of any $p \in \partial D \backslash \mathfrak{V}$.
 There is at least one pre-image of $p$ under $\hat{F}$.  Assume that
 there are three distinct pre-images, $q_1$, $q_2$ and $q_3$ say.  By the previous
 paragraph there exists an open
 disc $B(p;r)$ of $p$ and open neighbourhoods $V_i$ of $q_i$, $i=1,\ldots,3$ such that $\hat{F}$ has an
 extension to a
 biholomorphism of $V_i$ onto $B(p;r)$.  By Theorem \ref{th:zero_classification} we may take $r$ small enough
 that $B(p;r) \backslash \partial D$ consists of precisely two connected components, say $W$ and $W_*$.  By continuity
 of the extension of $\hat{F}$ we can furthermore choose $r$ small enough that $V_i$ is contained in some disc $B(q_i,r_i)$
 such that $r_i <1$ for $i=1,\ldots,3$.  For each $i$,
 the pre-image of either $W$ or $W_*$ is the connected component of $V_i \backslash \partial D$
 contained in $D$; call this $C_i$.  Thus the original map $F$ takes $C_1$, $C_2$ and $C_3$ each bijectively onto
 one of the sets $W$ and $W_*$.  This contradicts the fact that $F$ is one-to-one.  Thus there are at most
 two distinct pre-images of $p$.

  Now assume that $\hat{F}$ maps two arcs $A_i$ and $A_j$ onto $B_k$ say.
 For any $p \in B_k$, choosing a disk $B$ containing $p$ such that $B \cap D_2$ contains two connected components,
 since $F$ is orientation preserving we have that $\left. \hat{F} \right|_{A_i}$ and $\left. \hat{F} \right|_{A_j}$
 endow $B_k$ with opposite orientations.  This proves (1).

 The claims (3) and (4) follow immediately from the properties of the covering $\pi$.
 \end{proof}

 We may now prove Theorem \ref{th:convergence_and_nice_approximation}.
 \begin{proof}
  We prove both claims simultaneously.  Fix $\gamma_i$.  Let $\pi(\gamma_i) =  \hat{F}(A_j)$, say.

  Let $p \in \gamma_i$.  Since $\gamma_i$ is
  analytic there is a biholomorphism $G$ of an open set $U$ containing $p$ onto a disk $D$
  in the lower half plane such that $G(\gamma_i \cap U$) is an interval on the real line; let $J$
  be any compact sub-interval containing $G(p)$ in its interior and let $I$ be the
  compact subarc $G^{-1}(J)$ of $\gamma_i$.  The set $U$ can be chosen so that $\pi$ is a biholomorphism
  of $U$.
  We then have that  $q_2 \circ G^{-1} =
  \text{Re}(h)$ for some holomorphic $h$ on $D$, by the Schwarz reflection
  principle applied to $q_2$.    By the Cauchy-Riemann equations and conformal invariance of
  the integral (\ref{eq:star_change_o_coord})
  \[ \int_I q_1 \ast d q_2 =  \int_J q_1 \circ G^{-1} \ast d (q_2 \circ G^{-1}) =
  \text{Re} \int_J q_1 \circ G^{-1} \left( \frac{2}{i} \frac{\partial h}{\partial z} dz \right) \]
  which exists since $h$ is analytic on $J$ and $q_1$.  In particular, the integral exists on any
  compact sub-arc of $\gamma_i$.

  With notation as in Lemma \ref{le:distinguished_parametrization}, set $\hat{F}(e^{i\theta_p}) = \pi(p)$
  and observe that there is a sector
  \[ S_p= \{ re^{i\theta} \,:\, \theta_p - \epsilon \leq \theta \leq \theta_{p} + \epsilon
   \ \ \text{and} \ \ r_p \leq r \leq 1 \}  \]
  such that $\pi^{-1} \circ \hat{F}(S_p)$ is compactly contained in $U$ for one of the
  choices of $\pi^{-1}$.  Set $I_r = S_p \cap C_r$.  Since $G \circ \pi^{-1} \circ \hat{F}^{-1}$ is a holomorphic function
  of $z$ on an open neighbourhood of $S_p$, we have that
  \begin{align*}
   \lim_{r \nearrow 1}  \int_{\pi^{-1} \circ \hat{F}^{-1}(I_r)} q_1 \ast d q_2 & = \lim_{r \nearrow 1}
     \text{Re} \int_{G \circ \pi^{-1} \circ \hat{F}^{-1}(I_r)} q_1 \circ G^{-1} \frac{2}{i} \frac{\partial h}{\partial z} dz \\
     & =  \text{Re} \int_{G \circ \pi^{-1} \circ \hat{F}^{-1}(I_1)} q_1 \circ G^{-1} \frac{2}{i} \frac{\partial h}{\partial z} dz \\
    & =  \int_{\pi^{-1} \circ \hat{F}^{-1}(I_1)} q_1 \ast dq_2.
  \end{align*}

  On the other hand, if $p$ is an endpoint of a $\gamma_i$, it is a vertex.
  By Theorem \ref{th:zero_classification}
  we can find a map $\phi$ on the double cover in a neighbourhood $U$ of $p$
  so that $\pi^* \alpha$ has the form $w^n dw^2$ in a neighbourhood of $p$ (possibly $n=0$, if $p$ is
  a regular point of $\pi^*\alpha$).   Thus the trajectories map under $\phi$ to linear rays
  in the plane emanating from $0$.  We may assume that $\phi(U)$ is a disk centred at $0$, which is small enough that
  it contains no other trajectories than the rays.  Consider the connected component of $S=\phi(U) \cap \phi(\tilde{D}_2)$
  which is bounded by $\phi(\gamma_i)$ such that $\phi(\gamma_i)$ is positively oriented with respect to $S$ ($S$ is a
  radial segment of a disk).   It is bounded by an arc of a circle and another ray which must be
  $\phi(\gamma_j)$ for some $j$.
  At least one $\gamma_j$ must be such that $\phi(\gamma_j)$ is positively oriented with
  respect to the segment $S$, and we choose this one.
  Finally, choose a biholomorphism $H$ taking $S$ onto a half disk
  $\Omega = \{ z\,:\, |z| < r \ \ \text{and} \ \ \text{Im}(z) < 0\}$.
  By Carath\'eodory's theorem it extends to a homeomorphism of the boundary, and
  by composing with an automorphism of $\Omega$ we can arrange that $\phi(\gamma_i)$ and $\phi(\gamma_j)$
  map onto $(0,r)$ and $(-r,0)$.
  In summary, the map
  $G = H\circ \phi$ is a conformal bijection taking a connected component of $U \cap \tilde{D}_2$
  onto $\Omega$, and
  by Schwarz reflection $H$ has an analytic extension to a neighbourhood
  of $\gamma_i \cap U$ and to a neighbourhood of $\gamma_j \cap U$
  (we do not demand that these separate extensions agree).
  In particular, $q_2 \circ H$ extends continuously to $H(\gamma_i)$ and $H(\gamma_j)$ and equals $0$ there.
  By Schwarz reflection, $q_2 \circ H$ extends to a harmonic function on the full disk
  $\{ z \,:\, |z|<r \}$.   Choose a subarc $I$ of $\gamma_i \cap U$ with endpoint $p$,
  such that $J= H(I)$ is an interval $(0,s)$ or $(-s,0)$ where $s<r$.  Using change of
  variables
  \[   \int_I q_1 \ast d q_2 = \int_J q_1 \circ H \ast d (q_2 \circ H)     \]
  which converges since $q_2 \circ H$ has an analytic completion to $|z|<r$ and $q_1 \circ H$
  is continuous on the real line.
  Combining this with the convergence on compact subarcs of the interior, we have shown that the integral
  converges on $\gamma_i$.

  Assume now that $p$ is an initial endpoint of $\gamma_i$ with respect to its orientation. (The other
  case is similar so we omit it). As above, there is a sector
   \[S_p = \{ z=r e^{i\theta} \,:\, r_p \leq  r \leq 1 \ \ \text{and} \ \ \theta_j
    \leq \theta \leq \theta_j + \epsilon
  \}  \]
  such that $\pi^{-1} \circ \hat{F}(S_p)$ is compactly contained in $U$ for the relevant
  choice of $\pi^{-1}$.  Setting $I_r = S_p \cap C_r$ we obtain
  as above that
  \[ \lim_{r \nearrow 1} \int_{\pi^{-1} \circ \hat{F}(I_r)} q_1 \ast d q_2 = \int_{\pi^{-1} \circ \hat{F}(I_1)}
    q_1 \ast d q_2.  \]

  Set $C_r^j$ be the portion of $C_r$ between $\theta_j$ and $\theta_{j+1}$.  Since $\gamma_i$
  is compact, we have shown that for a single determination of $\pi^{-1}$ along $\hat{F}(A_i)$
  \[  \lim_{r \nearrow 1} \int_{\pi^{-1} \circ \hat{F} (C_r^j)} q_1 \ast d q_2 = \int_{\pi^{-1} \circ
  \hat{F}(A_j)} q_1 \ast d q_2  \]
  Since by part (4) of Lemma \ref{le:distinguished_parametrization} the set of such $\gamma_i$ is a
  complete set of maximal boundary arcs of $\tilde{D}_2$, this completes the proof.
 \end{proof}

 \begin{remark}  \label{re:needed_for_equality_case}
  In Nehari's paper \cite{Nehari_some_inequalities}, the problem of two-sided boundary arcs did not arise, since he
  assumed that the boundary of the domain was a finite number of closed disjoint analytic arcs.
  As we have seen, the details in this Appendix allow us to include extremal domains.
 \end{remark}

\end{section}

\end{document}